\documentclass{article}

\usepackage{amsmath,amssymb,mathrsfs,amsthm,ascmac,mathtools}
\usepackage{color} 
\usepackage{comment} 
\usepackage{multirow} 
\usepackage{multicol,vwcol} 

\newcommand{\veps}{\varepsilon} 
\newcommand{\eps}{\epsilon} 

\newcommand{\del}{\partial} 
\newcommand{\divv}[1][]{\text{div}_{#1}\,} 
\newcommand{\Divv}[1][]{\text{Div}_{#1}\,}
 
\newcommand{\loc}{\text{loc}} 

\newcommand{\tr}[1]{{#1^\mathsf{T}}}

\newcommand{\bb}[1]{\mathbb{#1}}
\newcommand{\cc}[1]{\mathcal{#1}}
\newcommand{\s}[1]{\mathscr{#1}}

\newcommand{\R}{\mathbb{R}}

\newcommand{\fb}[1]{{\bf #1}}

\newcommand{\bD}{{\bf D}}
\newcommand{\bI}{{\bf I}}
\newcommand{\bF}{{\bf F}}
\newcommand{\bG}{{\bf G}}
\newcommand{\bH}{{\bf H}}

\newcommand{\bM}{{\bf M}}
\newcommand{\bS}{{\bf S}}

\newcommand{\bU}{{\bf U}}
\newcommand{\bV}{{\bf V}}
\newcommand{\bW}{{\bf W}}
\newcommand{\bX}{{\bf X}}

\newcommand{\bff}{{\bf f}}
\newcommand{\bg}{{\bf g}}
\newcommand{\bh}{{\bf h}}
\newcommand{\bu}{{\bf u}}
\newcommand{\bv}{{\bf v}}

\newcommand{\bn}{{\bf n}}

\newcommand{\tp}{{\tilde{p}}}
\newcommand{\tq}{{\tilde{q}}}
\newcommand{\BR}{\bb{R}}

\newcommand{\wsp}[1][n]{{\bb{R}^{#1}}}
\newcommand{\hsp}[1][n]{{\bb{R}^{#1}_+}}

\newcommand{\FT}[1][]{\s{F}_{#1}} 
\newcommand{\IFT}[1][]{\s{F}^{-1}_{#1}}

\newcommand{\pres}{\mathfrak{p}} 
\newcommand{\press}{\mathfrak{q}} 
\newcommand{\Pres}{\mathfrak{P}} 
\newcommand{\nml}{{\bn}} 

\usepackage[dvipdfmx]{graphicx}
\usepackage{bmpsize}
\allowdisplaybreaks[4] 
\numberwithin{equation}{section}  
\mathtoolsset{showonlyrefs=true}

\theoremstyle{break}
\newtheorem{Def}{Definition}[section]
\newtheorem{Thm}{Theorem}[section]
\newtheorem{Prop}{Proposition}[section]
\newtheorem{Lem}{Lemma}[section]

\newtheorem{Rem}{Remark}[section]
\newtheorem{Ass}{Assumption}[section]

\newcommand{\lat}[1]{#1}
\newcommand{\DBC}[2]{#1}
\newcommand{\commm}[1]{}
\newcommand{\comm}[1]{}
\newcommand{\com}[1]{}

\newcommand{\cI}{\cc{I}}
\newcommand{\sopres}[1][q]{{\Pi_{#1}}}
\newcommand{\sopresz}[1][q]{{\Pi_{#1}^0}}
\newcommand{\sppresz}[1][q]{\cc{W}^1_{#1}(\Omega)}
\newcommand{\sppres}[1][q]{\sppresz[#1]+H^1_{#1}(\Omega)}
\newcommand{\sppreshz}[1][q]{\widehat{H}^1_{#1,0}(\hsp)}
\newcommand{\sppresh}[1][q]{\sppreshz[#1]+H^1_{#1}(\hsp)}
\newcommand{\bmpqz}[4][\delta]{{b_{#2}(#3,#4)}} 

\newcommand{\bmmpq}[1][m]{{\bmpqz{#1}{\tp_{#1}}{\tq_{#1}}}}
 
\newcommand{\TTT}{T}
\newcommand{\subti}[1]{\vspace{3mm}\underline{\bf #1).}\vspace{3mm}}
\newcommand{\tdh}[1][{1/2}]{{\del_t^{#1}}} 
\renewcommand{\hsp}{\BR^N_+}
\renewcommand{\bn}{\nu}

\title{\bf On the global well-posedness and decay of a free boundary problem 
of the Navier-Stokes equation in unbounded domains}
\author{Kenta Oishi 
\thanks{
Department of Applied Mathematics,  Waseda University, \endgraf
mailing address: 
Department of Mathematics, Waseda University,
Ohkubo 3-4-1, Shinjuku-ku, Tokyo 169-8555, Japan. \endgraf
e-mail address: kentaoishi@aoni.waseda.jp \endgraf
}
\enskip and 
\enskip Yoshihiro Shibata
\thanks{
Department of Mathematics,  Waseda University, \endgraf
Department of Mechanical Engineering and Materials Science,
University of Pittsburgh, USA \endgraf
mailing address: 
Department of Mathematics, Waseda University,
Ohkubo 3-4-1, Shinjuku-ku, Tokyo 169-8555, Japan. \endgraf
e-mail address: yshibata325@gmail.com \endgraf
Partially supported by Top Global University Project, JSPS 
Grant-in-aid for Scientific Research (A) 17H0109
}
}
\date{}

\begin{document}
\maketitle

\begin{abstract}
In this paper, we establish the unique existence and some decay properties of a global solution 
of a free boundary problem of the incompressible Navier-Stokes equations 
in $L_p$ in time and $L_q$ in space framework
in a uniformly $H^2_\infty$ domain $\Omega\subset \BR^N$ for $N\geq4$. 
We assume the unique solvability of the weak Dirichlet problem for the Poisson equation 
and the $L_q$-$L_r$ estimates for the Stokes semigroup. 
The novelty of this paper is that we do not assume the compactness of the boundary, 
which is essentially used in the case of exterior domains proved by Shibata \cite{Shiba17CIME}. 
The restriction $N\geq 4$ is required to deduce an estimate for the nonlinear term $\bG(\bu)$ arising from $\divv\bv=0$. 
However, we establish the results in the half space $\hsp$ for $N\geq 3$
by reducing the linearized problem to the problem with $\bG=\fb0$, 
where $\bG$ is the right member corresponding to $\bG(\bu)$. 

\end{abstract}


\section{Introduction} \label{sec:intro}

A free boundary problem for the viscous incompressible Navier-Stokes equations 
describes the motion of a fluid in time-dependent domains, 
such as a drop of water, 
an ocean of infinite extent and finite or infinite depth,
or liquid around a bubble. 
The present paper is concerned with the unique existence and decay 
of a global solution to these problems without taking account of surface tension. 
The mathematical problem is defined as finding 
a time-dependent domain $\Omega(t)$ 
in the $N$-dimensional Euclidean space $\BR^N$ where $t\geq0$ is the time variable, 
and the velocity field $\bv=\tr{(v_1(x,t),\ldots,v_N(x,t))}$, 
where $\tr{M}$ is the transposed $M$, 
and the pressure $\press=\press(x,t)$ satisfying the incompressible Navier-Stokes equation 
\begin{align}
\label{NS}
\left\{
\begin{array}{ll}
\del_t\bv + (\bv\cdot\nabla)\bv - \Divv\bS(\bv,\press) = 0, \quad
\divv \bv=0 & \text{in } \Omega(t), \ 0<t<T, \\
\bS(\bv,\press)\nml_t = 0, \quad \bv\cdot\nml_t=V_n
& \text{on } \del\Omega(t), \ 0<t<T, \\
\bv|_{t=0}=\bv_0 & \text{in } \Omega=\Omega(0) 
\end{array}
\right. 
\end{align}
for a given initial velocity field $\bv_0=\tr{(v_{01}(x),\ldots,v_{0N}(x))}$. 
Here, the initial domain $\Omega$ is a general uniformly $H^2_\infty$ domain.  
We denote the unit outer normal vector to the boundary $\del\Omega(t)$ by $\bn_t$, 
and the velocity of the evolution of $\del\Omega(t)$ by $V_n$. 
The stress tensor $\bS(\bv,\press)$ 
is given by $\bS(\bv,\press)=\mu\bD(\bv)-\press\bI$, 
where $\bD(\bu)$ is the doubled deformation tensor $\bD(\bv)$ 
whose $(j,k)$ component is defined by $\del_kv_j+\del_jv_k$ with $\del_j=\del/\del x_j$, 
$\mu>0$ is a positive constant representing the coefficient of viscosity, 
and $\bI$ is the $N\times N$ identity matrix. 
We set $\divv\bv=\sum_{j=1}^N\del_jv_j$ and, 
for an $N\times N$ matrix field $\bM$ whose $(j,k)$ component is $M_{jk}$, 
we define $\Divv \bM$ as the $N$ component vector the $j$-th component of which is $\sum_{k=1}^N\del_kM_{jk}$. 
In this paper, we establish the unique existence theorem of a solution globally in time 
and decay properties of the solution 
by assuming the $L_q$-$L_r$ estimates for the Stokes semigroup
and applying the maximal $L_p$-$L_q$ regularity 
for the time-shifted Stokes problem due to \cite{Shiba17CIME}. 
Moreover, we obtain the global well-posedness and decay properties 
in the half-space $\hsp$ with $N\geq3$. 

The free boundary problem \eqref{NS} 
has been studied extensively in the following two cases: 
\begin{enumerate}
\renewcommand{\labelenumi}{(\arabic{enumi})}
\setlength{\parskip}{0mm} 
\setlength{\itemsep}{0mm} 
\item the motion of an isolated liquid mass, and
\item the motion of the incompressible fluid occupying an infinite ocean.
\end{enumerate}
We mention the studies on the well-posedness globally in time and decay properties in order.   
In case (1), where the initial domain $\Omega$ is bounded, 
the unique existence of a global solution was established 
under the assumptions that the initial velocity $\bv_0$ is small 
and orthogonal to the rigid space $\{A\fb{x}+\fb{b}\mid A+\tr{A}=O\}$
in the frameworks 
by Solonnikov \cite{Sol87IANSSM}, in $L_p$ framework,
and by Shibata \cite{Shiba15DE} in $L_p$-$L_q$ framework. 
When surface tension is taken into account, 
the same result was also proved by Solonnikov \cite{Sol86ZNS} in $L_2$ framework 
under the same assumptions 
and the additional assumption that the domain $\Omega$ is close to a ball. 
In this case, the boundary condition should be 
\begin{align}
\bS(\bv,\press)\bn_t=c_\sigma\cc{H}\bn_t, \quad \bv\cdot\bn_t=V_n \quad \text{on } \del\Omega(t),
\end{align}
where $\cc{H}$ is the doubled mean curvature of $\del\Omega(t)$ 
and $c_\sigma>0$ is the coefficient of surface tension. 
This result was also obtained in H\"{o}lder spaces by Padula and Solonnikov \cite{PadSol02MQ} 
and, in $L_p$ in time and $L_q$ in space setting by Shibata \cite{Shiba18EECT}. 
In case (2), the domain is the layer-like domain given by the form
\begin{align}
\Omega(t)=\{x=\tr{(x',x_N)}\in\BR^N
\mid x'=\tr{(x_1,\cdots,x_{N-1})}\in\BR^{N-1},\,-b(x')<x_N<\eta(x',t)\},
\end{align}
with a free surface on the upper boundary $x_N=\eta(x',t)$
and a fixed bottom on the lower one $x_N=-b(x')$. 
Here, the boundary condition on the lower boundary is zero-Dirichlet: 
\begin{align}
\bv=0 \quad \text{ on } x_N=-b(x'). 
\end{align}
In this domain, in $L_2$ framework, 
the global well-posedness 
was established by Beale \cite{Bea84ARMA} with surface tension
and by Sylvester \cite{Syl90CPDE}, Tani and Tanaka \cite{TaniTana95ARMA}, 
and Guo and Tice \cite{GuoTice13APDE} without surface tension. 
Moreover, Beale and Nishida \cite{BeaNishi85NMS} 
and Hataya and Kawashima \cite{HataKawa09NA} 
proved some decay properties of the solution constructed in \cite{Bea84ARMA}. 
In $L_p$-$L_q$ framework, 
Saito \cite{Sai18JDE} showed the global well-posedness without surface tension. 
Hataya \cite{Hata09Ky} and Guo and Tice \cite{GuoTice13ARMA}  
obtained the unique existence and decay properties
of a global solution periodic in the horizontal direction in an $L_2$ framework.

The unique existence of a global solution to \eqref{NS} 
has been studied also in other domains. 
In exterior domains, Shibata \cite{Shiba17CIME} 
proved the global well-posedness without surface tension in $L_p$-$L_q$ setting. 
In the half-space, 
the global well-posedbess was obtained by Ogawa and Shimizu \cite{OgaShimi21EPE} 
without surface tension 
in $\dot{W}_1(0,\infty;\dot{B}^{-1+\frac{n}{p}}_{p,1}(\hsp))\cap L_1(0,\infty;\dot{B}^{1+\frac{n}{p}}_{p,1}(\hsp))$. 
This result and decay properties 
were developed by Saito and Shibata \cite{SaiShiba20arXiv}
with surface tension in $L_p$-$L_q$ framework. 
However, in the half-space and without taking account of surface tension, 
similar results in the $L_p$-$L_q$ framework have not been shown because of the noncompactness of the boundary $\del\Omega$. 
The analysis in $L_p$-$L_q$ framework seems more convenient 
than that in $L_1$-$\dot{B}^{-1+\frac{n}{p}}_{p,1}$ framework 
to address the lower derivative terms 
when we show the global well-posedness in other domains such as cylinder
, which remain as subject for further study.  
This is the key motivation of this paper.  

In the present paper, in $L_p$-$L_q$ framework, 
we establish the global well-posedness and decay properties of \eqref{NS} 
in the half space for a sufficiently small initial velocity $\bv_0$. 
Moreover, we obtain the same results in general domains 
by using the maximal $L_p$-$L_q$ regularity for the time-shifted Stokes problem, 
which was developed due to Shibata \cite{Shiba17CIME}, 
and by assuming the $L_q$-$L_r$ estimates for the Stokes semigroup. 
We also assume that the initial domain $\Omega\subset\BR^N$ is a uniformly $H^2_\infty$ domain 
and that the weak Dirichlet problem for the Poisson equation admits a unique solution, 
which are satisfied in the domains mentioned above. 
Because the domain is not compact, 
we cannot expect an exponential decay of the solution of the Stokes problem 
and the decay should be only of polynomial order. 
This forces us to restrict the dimension $N$ 
and the exponents $p$ and $q$ of the framework $L_p$-$L_q$, 
especially, when estimating a nonlinear term $\bG(\bu)$ arising from $\divv\bv=0$. 
In exterior domains, the global well-posedness was shown 
by relaxing the restriction from the compactness of the boundary $\del\Omega$. 
Nevertheless, we find the global well-posedness is valid even if the boundary $\del\Omega$ is not compact.
Moreover, in general domains, we obtain this result for $N\geq 4$ if the Stokes semigroup decays as in the half space. 
Moreover, we establish the same result for $N\geq3$ in the half space 
by some reduction of the Stokes problem to the case $\bG=0$, 
where $\bG$ is the right member corresponding to $\bG(\bu)$. 
Here, we take advantage of a good estimate obtained only in half space. (The further details, the reader is referred to Lemma \ref{Lem:diveq}.)  

The remainder of this paper is organized as follows. 
In the next section, we state our main results 
on the global well-posedness of \eqref{NS} and decay properties of the solution
in general domains with $N\geq 4$ and in the half-space with $N\geq 3$. 
Section \ref{sec:gen} is devoted to the proof of the results in general domains. 
The strategy is to prolong the local solution by use of the a priori estimate.
To obtain this estimate, 
we show an estimate for the Stokes problem in Subsection \ref{subsec:linesti:gen} 
and estimates for nonlinear terms in Subsection \ref{subsec:nonlinesti:gen}. 
In Section \ref{sec:half}, 
we show that the reduction mentioned above 
allows us to take $N=3$ in these results if $\Omega=\hsp$. 
The reduction will be performed in Subsection \ref{subsec:linesti:half} 
while Subsection \ref{subsec:ass:half} and Subsection \ref{subsec:nonlinesti:half} 
are devoted to the proof of $L_q$-$L_r$ estimates 
and estimates of the nonlinearities, respectively.
Finally, Section \ref{sec:conclusion} concludes the paper.


\section{Main results} \label{sec:}
In this section, we introduce notation and several functional spaces 
and then present the statements of our main results. 

\commm{Notation}
We denote the set of all natural numbers and real numbers by $\bb{N}$ and $\BR$, respectively. 
Let 
\begin{align}
<t>=(1+t^2)^{1/2}
\end{align}
for $t\in\BR$. 
Given a scalar function and an $N$-vector function $\bff=(f_1(x),\cdots,f_N(x))$, let 
\begin{align}
\nabla f &= (\partial_1f,\cdots,\partial_N f)^\mathsf{T},& 
\nabla \bff&= (\partial_jf_i)_{1\leq i,j\leq N} \\
\nabla^2f &= (\partial^\alpha f \mid |\alpha|=2),& 
\qquad \nabla^2\bff&= (\partial^\alpha f_j \mid |\alpha|=2, \ j=1,\cdots,N).
\end{align}
For a domain $D$, scalar functions $f,g$ and $N$-vector functions $\bff,\bg$, 
we define the normal part $\bff_\nml$ and tangential part $\bff_\tau$ of $\bff$ as 
\begin{align}
\bff_\nml=\nml\cdot\bff, \quad \bff_\tau=\bff-\bff_\nml\nml
\end{align} 
and let 
\begin{align}
(f,g)_D &= \int_D f(x)g(x) \,dx, 
& (\fb{f},\fb{g})_D &= \int_D \fb{f}(x)\cdot{\fb{g}}(x) \,dx, 
\end{align}
where $\fb{a}\cdot\fb{b} = \sum_{i=1}^N a_ib_i$ 
for $\fb{a}=(a_1,\cdots,a_N)^\mathsf{T}$ and $\fb{b}=(b_1,\cdots,b_N)^\mathsf{T}$. 
For a Banach space $X$ with a norm $\|\cdot\|_X$ and $d\in\bb{N}$, 
the $d$-product of $X$ is denoted by $X^d$, 
and the norm is expressed as $\|\cdot\|_X$ instead of $\|\cdot\|_{X^d}$ for brevity.  
The space of all bounded linear operators from $X$ to $X$ is denoted by $\cc{L}(X)$. 

Let $p,q\in[1,\infty]$, $m\in\bb{N}\cup\{0\}$ and $s\in\BR$, 
and let $D$ be a domain and $X$ a Banach space. 
The symbols $L_q(D;X)$, $H^m_q(D;X)$ and $B^{s}_{p,q}(D)$ 
denote the $X$-valued Lebesgue space, $X$-valued Sobolev spaces and Besov space, respectively, 
and we set $L_q(D)=L_q(D;\BR)$ and $H^m_q(D)=H^m_q(D;\BR)$. 
Note that $H^0_q(D;X)=L_q(D;X)$ and $H^0_q(D)=L_q(D)$. 
By $C_0^\infty(D)$, denote the set of all $C^\infty$ functions 
whose supports are compact and contained in $D$. 
We define the functional spaces 
\begin{align} \label{def:funcsp}
H^1_{q,0}(D)&=\{\theta\in H^1_q(D)\mid \theta|_{\del\Omega}=0\}, 
\\ \widehat{H}^1_{q,0}(D) &= \{\theta\in L_{q,\loc}(\overline{\Omega}) 
\mid \nabla\theta\in L_{q}(\Omega)^N, \ \theta|_{\del\Omega}=0\}, 
\\ \widehat{H}^{-1}_{q}(\Omega) &= \text{ dual of } \widehat{H}^{1}_{q',0}(\Omega), 
\\ H^{1/2}_{p}(\BR;X)&=\{G\in L_{p}(\BR;X)\mid \tdh G\in L_{p}(\BR;X)\}. 
\end{align}
Here, 
\begin{align}
\tdh{}f(t) = \IFT{}[|\tau|^{1/2}\FT{}[f](\tau)]
\end{align}
and the Fourier transform $\FT$ and its inverse transform $\IFT$ are defined by 
\begin{align} 
\s{F}[f](\tau) &= \int_{-\infty}^\infty e^{-it\tau} f(t) \,dt, & 
\s{F}^{-1}[f](t) &= \frac{1}{2\pi} \int_{-\infty}^\infty e^{i\tau t} f(\tau) \,d\tau 
\end{align}
for a function $f$ defined on $\BR$. 

To describe the nonlinear terms, 
for $m$-vector $\bu=(u_1,\cdots,u_m)$ and $n$-vector $\bv=(v_1,\cdots,v_n)$, 
we let $(\bu,\bv)=(u_1,\cdots,u_m,v_1,\cdots,v_n)$ 
and define $\bu\otimes\bv$ as the $mn$ vector 
whose $k$-th component is given by $u_{i}v_{j}$, 
where $(i,j)$ is the $k$-th couple of the set 
$\{(i',j')\mid 1\leq i'\leq m,\,1\leq j'\leq n\}$ 
in lexicographical order. 
Similarly, for $\ell\in\bb{N}\cup\{0\}$, 
we regard $\nabla^\ell\bu$ as the $N^\ell m$-vector 
the $k$-th component of which is given by $\del_{j_1}\cdots\del_{j_\ell}u_{i}$, 
where $(j_1,\cdots,j_\ell,i)$ is the $k$-th couple of the set 
$\{(j_1',\cdots,j_\ell',i')\mid 1\leq j_1',\cdots,j_\ell'\leq N,\,1\leq i'\leq m\}$ 
in lexicographical order. 
Then, for example, for an $\ell\times N^3mn$ matrix 
\begin{align}
\bV=(V_{k,(j_1,j_2,j_3,i,j)})_{1\leq k\leq \ell,1\leq j_1,j_2,j_3\leq N,1\leq i\leq m,1\leq j\leq m}, 
\end{align}
we can regard $\bV(\nabla^2\bu\otimes\nabla\bv)$ 
as the $\ell$-vector, with $k$-th component being 
\begin{align}
\sum_{j_1,j_2,i,j_3,j}V_{k,(j_1,j_2,j_3,i,j)}\del_{j_1}\del_{j_2}u_i\del_{j_3}v_j. 
\end{align}
Finally, the letter $C$ denotes generic constants 
and $C_{a,b,\cdots}$ stands for constants depending on the quantities $a,b,\cdots$. 
Both constants $C$ and $C_{a,b,\cdots}$ may vary from line to line.

\commm{Reduction to fixed domain}
We reduce the free boundary problem \eqref{NS} in the time-dependent domain $\Omega(t)$ 
to a quasilinear problem in the fixed domain $\Omega$.
Then, we provide our main results for the latter problem. 
To do so, we formulate the problem \eqref{NS} in Lagrange coordinates instead of Euler coordinates 
by employing the Lagrange transformation. 
\begin{gather} \label{eq:Ltrans}
x=y + \int_0^t \bu(y,s) \, ds \equiv \bX_\bu(y,t) \\ \notag
\bu = (u_1(y,t),\cdots,u_N(y,t))=\bv(\bX_\bu(y,t),t), \quad \pres(y,t)=\pi(\bX_\bu(y,t),t). 
\end{gather}
By the argument in \cite[Appendix A]{ShibaShimi07DIE}, 
the functions $\bu,\pres$ satisfies the following equation. 
\begin{align} \label{LNS} 
\left\{
\begin{array}{ll}
\del_t\bu - \Divv\bS(\bu,\pres) = \bff(\bu), \quad \divv \bu=g(\bu)=\divv \bg(\bu) 
& \text{in } \Omega\times(0,T), \\
\bS(\bu,\pres)\nml = \bh(\bu)\bn & \text{on } \del\Omega\times(0,T), \\
\bu|_{t=0}=\bv_0 & \text{in } \Omega 
\end{array}
\right.
\end{align}
where the nonlinearities $\bff(\bu)$, $g(\bu)$, $\bg(\bu)$ and $\bh(\bu)$ are defined by
\begin{align} \label{def:nonlin}
\begin{aligned}
\bff(\bu) &= \bV^1\left( \int_0^t \nabla\bu\,ds \right)(\del_t\bu,\nabla^2\bu) 
+ \bW\left( \int_0^t \nabla\bu\,ds \right) (\int_0^t \nabla^2\bu \,ds \otimes \nabla\bu), \\
\bg(\bu) &= \bV^2\left( \int_0^t \nabla\bu\,ds \right)\bu, \\ 
g(\bu) &= \bV^3\left( \int_0^t \nabla\bu\,ds \right)\nabla\bu, \quad 
\bh(\bu) = \bV^4\left( \int_0^t \nabla\bu\,ds \right)\nabla\bu, 
\end{aligned}
\end{align}
with some matrix-valued polynomials $\bV^1,\bV^2,\bV^3,\bV^4$ and $\bW$ 
with $\bV^i(\fb O)=\fb O$.  
The symbol $\fb O$ stands for the zero matrix. 

\commm{Assumptions}
To establish the global well-posedness of \eqref{LNS} and the decay of the solution, 
the appropriate decay properties of the solution must be proven for 
the linearlized problem associated with \eqref{LNS}, 
which is called the Stokes initial value problem. 
\begin{align} \label{S}
\left\{
\begin{array}{ll}
\del_t\bU - \Divv\bS(\bU,\Pres) = \bF, \quad
\divv \bU=G=\divv\bG & \text{in } \Omega\times(0,T), \\
\bS(\bU,\Pres)\nml = \bH & \text{on } \del\Omega\times(0,T), \\
\bU|_{t=0}=\bv_0 & \text{in } \Omega. 
\end{array}
\right. 
\end{align}
Nevertheless, the decay properties have not been developed in general domains.  
In this paper, we focus on the case 
\begin{align}
\{\lambda\in\bb{C}\setminus\{0\}\mid |\arg\lambda|<\pi-\veps\}\subset \rho(A_q)
\text{ but } 0\not\in\rho(A_q)
\end{align}
with $\veps\in(0,\pi/2)$, 
where $\rho(A_q)$ is the resolvent set of the Stokes operator $A_q$, and we 
assume the $L_q$-$L_r$ estimates for the Stokes semigroup. 
To obtain decay properties, 
we consider a time-shifted problem, \eqref{SS} below, whose solution decays sufficiently fast, 
and then compensate it 
by estimating the difference of solutions to \eqref{S} and the time-shifted problem 
from the $L_q$-$L_r$ estimates.

Below, we \lat{state} the assumptions of our main theorem in order. 
We begin with the assumption that the domain $\Omega$ is a uniform $H^{2}_{\infty}$ domain. 

\comm{Ass1: $H^{2-1/r}_{r}$}
\begin{Ass} \label{Ass:unifdom}
There exist positive constants $\alpha$, $\beta$ and $K$ 
such that for any $x_0=(x_{0,1},\cdots,x_{0,N}) \in \del\Omega$, 
there exist a coodinate number $j$ and a function $h\in H^{2}_\infty(B_\alpha'(x_0'))$
with $\|h\|_{H^{2}_{\infty}(B_\alpha'(x_0'))}\leq K$ satisfying 
\begin{align}
\Omega\cap B_\beta(x_0) 
&= \{ x\in\R^N \mid x_j>h(x'), \ \forall x'\in B_\alpha'(x_0') \} \cap B_\beta(x_0),  
\\ \del\Omega\cap B_\beta(x_0) 
&= \{ x\in\R^N \mid x_j=h(x'), \ \forall x'\in B_\alpha'(x_0') \} \cap B_\beta(x_0),  
\end{align} 
where 
\begin{align}
x'=(x_1,\cdots,x_{j-1},x_{j+1},\cdots,x_N), 
\ x_0'=(x_{0,1},\cdots,x_{0,j-1},x_{0,j+1},\cdots,x_{0,N}), 
\\ B_\alpha'(x_0')=\{ x'\in\R^{N-1} \mid |x'-x_0'|<\alpha \}, 
\ B_\beta(x_0)=\{ x\in\R^{N} \mid |x-x_0|<\beta \}.  
\end{align}
\end{Ass} 
\comm{Ass2: US wDNp}
We apply the unique solvability 
of the weak Dirichlet\DBC{}{-Neumann} problem for the Poisson equation to reduce the linearized problem to a problem without the divergence condition. 
This unique solvability is known only for $q=2$. Hence, we assume it in the present paper. 
This assumption is reasonable 
because the resolvent estimate for the Stokes resolvent problem cannot be obtained 
if the unique solvability does not hold. (See \cite[Remark 1.7]{Shiba13MFM}.) 
\begin{Ass} \label{Ass:wDp} 
The following assertion holds
for $\cc{W}^1_q(\Omega) = \widehat{H}^{1}_{q,\DBC{0}{\Gamma}}(\Omega)$ 
or $\cc{W}^1_q(\Omega) = {H}^{1}_{q,\DBC{0}{\Gamma}}(\Omega)$. 
For any $q\in(1,\infty)$ and $\bff\in L_{q}(\Omega)^N$, the problem 
\begin{align} \label{wDp}
(\nabla\theta,\nabla\varphi)_\Omega = (\bff,\nabla\varphi)_\Omega 
\text{ for all } \varphi\in\cc{W}^1_{q'}(\Omega)
\end{align}
admits a unique solution $\theta\in\cc{W}^1_q(\Omega)$ satisfying the estimate
\begin{align} \label{esti:wDp}
\|\nabla\theta\|_{L_{q}(\Omega)} \leq C\|\bff\|_{L_{q}(\Omega)}. 
\end{align}  
Moreover, for any $r\in(1,\infty)$, if $\bff\in L_r(\Omega)^N$ as well as $\bff \in L_q(\Omega)^N$, then, 
$\bu$ satisfies 
\begin{align} 
(\nabla\theta,\nabla\varphi)_\Omega = (\bff,\nabla\varphi)_\Omega, 
\text{ for all } \varphi\in\cc{W}^1_{r'}(\Omega)
\end{align}
and the estimate $\|\nabla\theta\|_{L_{r}(\Omega)} \leq C\|\bff\|_{L_{r}(\Omega)}$. 
\end{Ass}
\comm{Rem1: wDNp}
\begin{Rem} \label{Rem:wDNp}
The unique existence of the weak Dirichlet problem \eqref{wDp} is obtained 
in the half-space, bounded domains, exterior domains, 
perturbed half-spaces, and layer domains. 
For details and more examples of domains in which the problem \eqref{wDp} is uniquely solvable, 
see \cite[Example 1.6]{Shiba13MFM}. 
\end{Rem}

Define the solution operator $\sopresz$ of the weak Dirichlet problem \eqref{wDp} by 
\begin{align} \label{def:sopres0}
\sopresz:L_{q}(\Omega)\to\cc{W}^1_q(\Omega):\sopresz\bff = \theta.
\end{align}
Note that $\sopresz\bff=\sopresz[r]\bff$ if $\bff\in L_{q}(\Omega)^N\cap L_{r}(\Omega)^N$. 

\com{Reduction to RS} 
We now introduce the Stokes semigroup 
by following the arguments in \cite[Section 4]{GruSol91MS}, 
see also \cite[p. 159, 160]{ShibaShimi08RAM}.
We assume that Assumptions \ref{Ass:unifdom} and \ref{Ass:wDp} hold. 
We consider the Stokes initial value problem 
\begin{align} \label{Sv0}
\left\{
\begin{array}{ll}
\del_t\bU - \Divv\bS(\bU,\Pres) = \bF, \quad
\divv \bU=0 & \text{in } \Omega\times(0,\infty), \\
\bS(\bU,\Pres)\nml = 0 & \text{on } \del\Omega\times(0,\infty), \\
\bU|_{t=0}=\bv_0 & \text{in } \Omega 
\end{array}
\right. 
\end{align}
for $\bF\in L_{p}(0,\infty;J_{q}(\Omega))$ and $\bv_0\in J_q(\Omega)$, 
where $J_q(\Omega)$ is the solenoidal space 
\begin{align} \label{def:solesp}
J_q(\Omega) = \{ \bU\in L_q(\Omega)^N \mid (\bU,\nabla\varphi)_\Omega=0 
\text{ for any } \varphi\in\cc{W}^1_{q'}(\Omega) \}. 
\end{align}
Then, $\divv\bF(t)=0$ a.e. $t\in(0,\infty)$ and $\divv\bv_0=0$. 
By multiplying $\divv$ to the first equation, 
by the normal component of the boundary condition 
and by applying $\divv\bU=0$, 
we obtain a system for the pressure $\Pres$, as given below.
\begin{align} \label{pres}
\left\{\begin{array}{ll}
\Delta{\Pres}=0 & \text{in } \Omega, \\ 
\Pres=2\mu\del_\nml \bU_\nml-\text{div}_{}\,{\fb{U}}& \text{on } \partial\Omega, 
\end{array}\right.
\end{align}
where $\bff_\nml=\nml\cdot\bff$, $\del_\nml f=\nml\cdot\nabla f$ 
for given functions $\bff=(f_1(x),\cdots,f_n(x))$ and $f=f(x)$. 
The solution operator of this system is given by 
\begin{align} \label{def:sopres}
\quad 
\sopres:H^{2}_{q}(\Omega)\to\cc{W}^1_q(\Omega):
\sopres\bU=(2\mu\del_\nml\bU_\nml-\divv\bU)-\sopresz\nabla(2\mu\del_\nml\bU_\nml-\divv\bU), 
\quad 
\end{align}
where the operator $\sopresz$ is defined by \eqref{def:sopres0}.  
In fact, $\theta=\Pres-(2\mu\del_\nml\bU_\nml-\divv\bU)$ obeys
\begin{align}
\left\{\begin{array}{ll}
\Delta{\theta}=-\Delta(2\mu\del_\nml\bU_\nml-\divv\bU) & \text{in } \Omega, \\ 
\theta=0& \text{on } \partial\Omega, 
\end{array}\right.
\end{align}
whose weak formulation is given by \eqref{wDp} with 
$\bff=-\nabla(2\mu\del_\nml\bU_\nml-\divv\bU)$. 
\comm{RSR, Stokes op, Stokes sg}
Then, the Stokes initial value problem \eqref{Sv0} can be reduced to the problem 
\begin{align} \label{RSv0}
\left\{
\begin{array}{ll}
\del_t\bU - \Divv\bS(\bU,\sopres(\bU)) = \bF & \text{in } \Omega\times(0,\infty), \\
\bS(\bU,\sopres(\bU))\nml = 0 & \text{on } \del\Omega\times(0,\infty), \\
\bU|_{t=0}=\bv_0 & \text{in } \Omega. 
\end{array}
\right. 
\end{align}
Note that the second equation $\text{div}_{}\,{\fb{u}}=0$ can be recovered 
by the uniqueness of solutions to the initial value problem 
for the heat equation subject to the Dirichlet boundary condition obeyed by $\text{div}_{}\,{\fb{u}}$. (see \cite[p. 243]{GruSol91MS}. )
We now define the Stokes operator \lat{on $J_q(\Omega)$} as 
\begin{align} \label{def:Aq}
\begin{aligned}
D(A_q)&=\{ \fb{U}\in J_q(\Omega)\cap{H^2_q(\Omega)}^N 
\mid \fb{S}(\fb{U},\Pi(\fb{U}))\nu=0 \text{ on } \partial\Omega \}, 
\\ A_q{\fb{U}}&= -\text{Div}_{}\,{\fb{S}}(\fb{U},\Pi(\fb{U})). 
\end{aligned}
\end{align}
Then, \eqref{RSv0} is rewritten to the Cauchy problem 
\begin{align} \label{Cv0}
\del_t\bU+A_q\bU=\bF, \quad \bU|_{t=0}=\bv_0. 
\end{align}
Note that, for any $q,r\in(1,\infty)$, $A_q\bU=A_r\bU$ if $\bU\in D(A_q)\cap D(A_r)$. 

\comm{Prop1: generation of Stokes sg} 
The following proposition on the generation of the Stokes semigroup 
is guaranteed by \cite[Theorem 2.5]{Shiba14DIE}. 
For the details of the proof, the reader is referred to \cite[Lemma 3.7]{ShibaShimi08RAM}.  
\begin{Prop}[{\cite{Shiba14DIE}}] 
Assume that Assumptions 1 and 2 hold. 
Then, the Stokes operator $A_q$ generates an analytic semigroup $\{e^{-tA_q}\}_{t\geq0}$ 
of class $C^0$ on $J_q(\Omega)$ for $1<q<\infty$.
\end{Prop}

We often write $e^{-tA_q}\bff$ even for $\bff\in J_r(\Omega)$ instead of $e^{-tA_r}\bff$ 
because $e^{-tA_q}\bff=e^{-tA_r}\bff$
for $q,r\in(1,\infty)$ and $\bff \in J_{q}(\Omega)\cap J_{r}(\Omega)$. 
In fact, by repeating the argument in \cite{Shiba13MFM} 
in $J_{q}(\Omega) \cap J_{r}(\Omega)$ instead of $J_{q}(\Omega)$,
we obtain $(\lambda+A_q)^{-1}\bff=(\lambda+A_r)^{-1}\bff$. 
Then, the formula 
\begin{align}
e^{-tA_q}\bff = \frac{1}{2\pi i}\int_\Gamma (\lambda+A_q)^{-1}\bff \,ds, 
\end{align} 
concludes $e^{-tA_q}\bff=e^{-tA_r}\bff$. 

\comm{Lq-Lr esti}

The $L_q$-$L_r$ estimates are stated as follows. 
Because the decay rate changes according to the domain, (see Remark \ref{Rem:LqLresti}) 
we consider the general rate 
and, in the statement of the main theorem, 
state the type of rate needed to obtain the global well-posedness. 
\begin{Def} \label{Def:LqLrEsti}
Let the decay rate $\sigma_m(q,r)$ be a function defined 
for $m=0,1,2$ and $q,r\in(1,\infty]$ with $q\leq r$. 
We say that the $L_q$-$L_r$ estimates hold for the decay rate $\sigma_m(q,r)$ if, 
for $(q,r)$ satisfying $1<q\leq r\leq\infty$ and $q\ne\infty$, 
there exists $C=C(q,r)>0$ such that 
\begin{align} \label{esti:LqLr}
\begin{aligned}
\| (\del_te^{-tA_q}\bff,\nabla^2e^{-tA_q}\bff) \|_{L_{r}(\Omega)} 
&\leq Ct^{-\sigma_2(q,r)}\|\bff\|_{L_{q}(\Omega)} \quad (r\ne\infty) \\ 
\| \nabla^m e^{-tA_q}\bff \|_{L_{r}(\Omega)} 
&\leq Ct^{-\sigma_m(q,r)}\|\bff\|_{L_{q}(\Omega)} \quad (m=0,1) 
\end{aligned}
\end{align}
for $t\geq1$ and $\bff\in J_{q}(\Omega)$. 
\end{Def}

\begin{Rem} \label{Rem:LqLresti}
\begin{enumerate}
\renewcommand{\labelenumi}{(\arabic{enumi})}
\item 
In $\bb{R}^N$ and $\bb{R}^N_+$, 
the $L_q$-$L_r$ estimates hold for the decay rate 
\begin{align} \label{eq:decayLpLq}
\sigma_m(q,r)=\frac{N}{2}\left(\frac{1}{q}-\frac{1}{r}\right)+\frac{m}{2}.
\end{align}
The second inequality of \eqref{esti:LqLr} 
was studied in \cite[equation (2.3)]{Kato84MZ} in $\bb{R}^N$, 
and, in $\bb{R}^N_+$, is proven in Subsection \ref{subsec:ass:half}  
below from the resolvent estimates for the resolvent Stokes problem 
provided by to Shibata and Shimizu \cite{ShibaShimi03DIE}. 
The first inequality is obtained as in Subsection \ref{subsec:ass:half}. 
\item 
In exterior domain, the $L_q$-$L_r$ estimates hold for the decay rate 
\begin{align} \label{eq:decayLpLqext}
\sigma_m(q,r)=\left\{ \begin{array}{ll}
\frac{N}{2}\left(\frac{1}{q}-\frac{1}{r}\right)+\frac{m}{2} & m=0,1, \\ 
\min \left\{ \frac{N}{2}\left(\frac{1}{q}-\frac{1}{r}-\delta_0\right)+1, 
\frac{N}{2q} + \frac{1}{2} \right\} & m=2
\end{array} \right.
\end{align}
for sufficiently small $\delta_0>0$. 
The second inequality was proven by Shibata \cite[Theorem 1]{Shiba18AA} 
and first one is shown as in Subsection \ref{subsec:ass:half} 
from the resolvent estimates by Shibata \cite[Theorem 2]{Shiba18AA}. 
\end{enumerate}
\end{Rem}

\comm{SMR}

The sufficiently fast decay of the solution to the time-shifted Stokes problem
\begin{align}
\label{SS}
\left\{
\begin{array}{ll}
\del_t\bU + \lambda_0\bU - \Divv\bS(\bU,\Pres) = \bF, \
\divv \bU=G=\divv\bG & \text{in } \Omega\times(0,T), \\
\bS(\bU,\Pres)\nml = \bH & \text{on } \del\Omega\times(0,T), \\
\bU|_{t=0}=\bv_0 & \text{in } \Omega
\end{array}
\right. 
\end{align}
is justified by \cite[equation (3.600)]{Shiba17CIME}, 
which is valid for general domains satisfying Assumptions \ref{Ass:unifdom} and \ref{Ass:wDp}. 
To provide a statement of it, we define the space for initial velocity by 
\begin{align}
\cc{D}_{q,p}(\Omega) = (J_q(\Omega),D(A_q))_{1-1/p,p} \subset B^{2(1-1/p)}_{q,p}(\Omega)
\end{align}
for $p,q\in(1,\infty)$, where $(\cdot,\cdot)_{1-1/p,p}$ is real interpolation functor. 

\begin{Rem} \label{Rem:Dqp}
The space $\cc{D}_{q,p}(\Omega)$ is characterized as follows. 
(see \cite[Lemma 2.4]{Ste06MFM}) 
\begin{align}
\cc{D}_{q,p}(\Omega) = \left\{ \begin{array}{l}
\{\bv_0\in J_q(\Omega) \cap B^{2(1-1/p)}_{q,p}(\Omega) \mid [\bD(\bv_0)\nml]_\tau = 0 \}
\\ \qquad \textup{if } 2(1-1/p)-1/q>1, 
\\ J_q(\Omega) \cap B^{2(1-1/p)}_{q,p}(\Omega) \quad \textup{if } 2(1-1/p)-1/q<1. 
\end{array} \right.
\end{align}
\end{Rem}

\begin{Thm}[{\cite{Shiba17CIME}}] \label{Thm:SMR}
Let $1<p,q<\infty$ and $T\in(0,\infty]$.  
Assume Assumptions 1 and 2. 
For the right members $\bF,\bG,G,\bH,\bv_0$, 
assume that $\bF,\bv_0$, and some extensions $\bG_b,G_b,\bH_b$ 
respectively of $<t>^b\bG,<t>^bG,<t>^b\bH$ satisfy 
\begin{align} \label{eq:DataClassMR}
\begin{gathered}
\bv_0\in \cc{D}_{q,p}(\Omega), 
\quad <t>^b\bF\in L_{p}(0,T;L_{q}(\Omega)^N), 
\quad \bG_b \in H^1_{p}(\BR;L_{q}(\Omega)^N), 
\\ G_b \in H^{1/2}_{p}(\BR;L_{q}(\Omega))\cap L_{p}(\BR;H^{1}_{q}(\Omega)), 
\quad \bH_b \in H^{1/2}_{p}(\BR;L_{q}(\Omega)^N)\cap L_{p}(\BR;H^{1}_{q}(\Omega)^N), 
\end{gathered}
\end{align}
the compatibility condition
\begin{align} \label{eq:DataCCMR}
(G_b(t),\varphi)_\Omega = (\bG_b(t),\nabla\varphi)_\Omega 
\text{ a.e. $t\in\BR$ for any } \varphi\in\cc{W}^{1}_{q'}(\Omega) 
\end{align}
and $(\bG,G,\bH)|_{t=0}=(\fb0,0,\fb0)$. 
Then, the problem \eqref{SS} admits unique solutions 
\begin{gather}
\bU \in H^{1}_{p}(0,T;L_{q}(\Omega)^N) \cap L_{p}(0,T;H^{2}_{q}(\Omega)^N), \ 
\Pres \in L_{p}(0,T;\sppres). 
\end{gather}
Moreover, for $b\geq0$, the solution satisfies the following estimate. 
\begin{align} \label{esti:MR}
\begin{aligned}
&\|<t>^b\del_t\bU\|_{L_{p}(0,T;L_{q}(\Omega))}+\|<t>^b\bU\|_{L_{p}(0,T;H^2_{q}(\Omega))}
+\|\nabla\Pres\|_{L_{p}(0,T;L_{q}(\Omega))} 
\\ & \begin{aligned}
\leq C (&\|<t>^b\bF\|_{L_{p}(0,T;L_{q}(\Omega))}
+\|(\bG_b,\del_t\bG_b)\|_{L_{p}(\BR;L_{q}(\Omega))}
\\ & +\|\tdh (G_b,\bH_b)]\|_{L_{p}(\BR;L_{q}(\Omega))}
+\|(G_b,\bH_b)\|_{L_{p}(\BR;H^1_{q}(\Omega))}
+\|\bv_0\|_{B^{2(1-1/p)}_{q,p}(\Omega)}), 
\end{aligned}
\end{aligned}
\end{align}
where the constant $C$ is independent of $T$ and dependent on $b$. 
\end{Thm}

\commm{Main theorem}


The following theorem on the global well-posedness of \eqref{LNS} in general domain 
is one of our main results. 
\begin{Thm} \label{Thm:GWPgen}
Let $2<p<\infty$, $1<q_0<N<q_2<\infty$ and $b>1/p'$. 
Assume that $\Omega$ is a uniformly $H^{2}_{\infty}(\Omega)$ domain 
and that the weak Dirichlet problem is uniquely solvable in $\cc{W}^1_q(\Omega)$ for $q\in(1,\infty)$ 
as stated in Assumptions 1 and 2. 
Also, assume that the $L_q$-$L_r$ estimates hold for a decay rate $\sigma_m(q,r)$ 
defined for $m=0,1,2$ and $(q,r)\in(1,\infty)$ with $q\geq r$ and satisfying the following conditions.   
\begin{itemize}
\item[\textup{(C1)}] 
$\sigma_m(q_0,r)$ and $\sigma_0(q_2,r)$ is non-negative and non-decreasing with respect to $m$ and $r$, 
\item[\textup{(C2)}] 
$\sigma_0(q_0,q_{04})>b+\frac{1}{p}$, $\sigma_1(q_0,q_{03})>1$ 
for some $q_{03},q_{04}\in[q_0,q_2]$ 
with $\frac{1}{q_0}=\frac{1}{q_{03}}+\frac{1}{q_{04}}$. 
\end{itemize}
Then there exists $\epsilon>0$ such that 
for any $\bv_0 \in \bigcap_{i=0,2} \cc{D}_{q_i,p}(\Omega) 
\subset \bigcap_{i=0,2} B^{2(1-1/p)}_{q_i,p}(\Omega)$ 
with smallness $\sum_{i=0,2}\|\bv_0\|_{B^{2(1-1/p)}_{q_i,p}(\Omega)}\leq\epsilon$, 
the transformed problem \eqref{LNS} admits unique solutions 
\begin{align} \label{eq:solclass}
\bu\in H^{1}_{p}(0,\infty;L_{q_2}(\Omega)^N)\cap L_{p}(0,\infty;H^{2}_{q_2}(\Omega)^N), \ 
\pres\in L_{p}(0,\infty;\cc{W}^{1}_{q_2}(\Omega)+H^{1}_{q_2}(\Omega))
\end{align}
possessing the estimate $[\bu]_{(0,\infty)}\leq C\epsilon$. 
Here, for an interval $(a,b)$, we let 
\begin{align} \label{def:Normofu}
\begin{aligned}
[\bu]_{(a,b)}
&=\sup_{(\tp_2,\tq_2)\in I_2} 
\|(1+t)^{b_2(\tp_2,\tq_2)}\del_t\bu \|_{L_{\tp_2}(a,b;L_{\tq_2}(\Omega))}
\\ &+\sum_{m=0,1,2}\sup_{(\tp_m,\tq_m)\in I_m} 
\|(1+t)^{\bmmpq}\nabla^m\bu \|_{L_{\tp_m}(a,b;L_{\tq_m}(\Omega))}, 
\end{aligned}
\end{align}
where the power $b_m(\tp_m,\tq_m)$ of the weight is defined as 
\begin{align} \label{def:bmpq}
\begin{aligned}
b_m(\tilde p_m,\tilde q_m)=
\left\{ \begin{array}{ll}
\min\{\sigma_m(q_0,\tilde q_m)-\frac{1}{\tilde p_m}-\delta,b\} & (\tp_m<\infty)
\\ \min\{\sigma_m(q_0,\tilde q_m),b\} & (\tp_m=\infty)
\end{array} \right.
\end{aligned}
\end{align}
with $\delta>0$ satisfying
\begin{align} \label{eq:deltaCondi}
\delta < \min\{\sigma_0(q_0,q_{04})-(b+1/p), \ \sigma_1(q_0,q_{03})-1, \ b-1/p' \}
\end{align}
and the index set $I_m$ is the set of all 
$(\tp_m,\tq_m)\in\{p,\infty\}\times[q_{0},\infty]$ 
satisfying
\begin{align} \label{eq:exp_restriction} 
\begin{aligned}
&(\tp_2,\tq_2)\in\{p\}\times[q_0,q_2], 
\\ &(\tp_1,\tp_1)\in(\{p\}\times[q_0,\infty])\cup(\{\infty\}\times[q_0,q_2]), 
\\ &(\tp_0,\tp_0)\in(\{p\}\times[q_0,\infty])\cup(\{\infty\}\times[q_0,\infty]). 
\end{aligned}
\end{align}
Moreover, the solution has the decay property 
\begin{align} \label{eq:decay} 
\|\nabla^m\bu(t)\|_{L_{r}(\Omega)} 
=O(t^{-\min\left\{\sigma_m(q_0,r),b\right\}})
\end{align}
for all 
$r\in[q_0,\infty]$ ($m=0$) and $r\in[q_0,q_2]$ ($m=1$). 
\end{Thm}

\comm{Rem of Main theorem}
\begin{Rem} \label{Rem:GWPgen}
In the exterior domain, Shibata \cite{Shiba17CIME} 
developed the global well-posedness for the dimension $N\geq3$ 
by the compactness of the boundary $\del\Omega$ of the domain. 
In fact, he changed the transformation from \eqref{eq:Ltrans} 
so that the supports of the nonlinear terms lay near $\del\Omega$.
Then, the supports are bounded thanks to the compactness of $\del\Omega$.
This improves the decay of the nonlinear terms from the $L_q$-$L_r$ estimates 
by lifting up the exponent $r$ of $L_r(\Omega)$. 
In this paper, we make full use of the decay arising from the derivative 
($m/2$ appearing in \eqref{eq:decayLpLq} or \eqref{eq:decayLpLqext}) instead 
and obtain the global well-posedness. 
\end{Rem}

\comm{Thm: GWP half-space}
The condition \textup{(C2)} in Theorem \ref{Thm:GWPgen} requires us to take $N\geq 4$ 
even if the decay rate $\sigma_m(\tp,\tq)$ is as fast as that in the half-space, or more specifically satisfies \eqref{eq:decayLpLq}.  
This condition is required to estimate 
the nonlinear term $\bG(\bu)$ arising from $\divv\bv=0$, see Remark \ref{Rem:C2}. 
However, by reducing the Stokes equation \eqref{S} to the problem with $(\bG,G)=(\fb0,0)$ 
(see Subsec. \ref{subsec:linesti:half}), 
we establish the results also for $N=3$ in the half-space. 

\begin{Thm} \label{Thm:GWPha}
Let $2<p<\infty$, $1<q_0<N<q_2<\infty$, $b>1/p'$. 
Assume  
\begin{align}
b+\frac{1}{p}<\frac{N}{2q_0}.  
\end{align}
Then, the global well-posedness and decay property stated in Theorem \ref{Thm:GWPgen}
hold with $\Omega=\BR^N_+$ 
and with $\delta>0$ in the definition \eqref{def:bmpq} of $\bmmpq$ being 
\begin{align} \label{def:deltahalf}
\delta<\frac12\left(\frac{N}{2q_0}-(b+\frac1p)\right). 
\end{align}
\end{Thm}


\section{Proof of Theorem \ref{Thm:GWPgen}} \label{sec:gen}

In this section, we develop the global well-posedness and decay properties of the solution
of the transformed problem \eqref{LNS} in general domains 
stated in Theorem \ref{Thm:GWPgen}. 

The strategy to prove the global well-posedness 
is to prolong the local solution by proving an a priopri estimate. 
The unique existence of the local solution, which is stated as follows, 
is guaranteed by a similar argument to that in \cite[Theorem 2.4]{Shiba15DE}. 
 
\begin{Thm}[{\cite{Shiba15DE}}]] \label{Thm:LWP}
Let $2<p<\infty$, $N<q<\infty$ and $T>0$. Assume Assumptions 1 and 2 hold. 
Then, there exists an $\epsilon>0$ depending on $T$ 
such that, for any $\bv_0\in\cc{D}_{q,p}(\Omega)\subset B^{2(1-1/p)}_{q,p}(\Omega)$ 
with smallness condition $\|\bv_0\|_{B^{2(1-1/p)}_{q,p}(\Omega)}\leq \epsilon$, 
the quasilinear problem \eqref{LNS} admits a unique solution 
\begin{align}
\bu\in H^{1}_{p}(0,T;L_{q}(\Omega))\cap L_{p}(0,T;H^{2}_{q}(\Omega)), \ 
\pres\in L_{p}(0,T;\sppres)
\end{align}
possessing the estimate
\begin{align} \label{esti:LWP}
\|\del_t\bu\|_{L_{p}(0,T;L_{q}(\Omega))}+\|\bu\|_{L_{p}(0,T;H^2_{q}(\Omega))}
+\|\nabla\pres\|_{L_{p}(0,T;L_{q}(\Omega))}
\leq C\epsilon, 
\end{align} 
with some positive constant $C>0$ independent of $T$ and $\epsilon$. 
\end{Thm}

Then, by the same argument as in, e.g., \cite[Subsec. 3.8.6]{Shiba17CIME}, 
it suffices to prove that the a priori estimate 
\begin{align} \label{esti:apiori}
[\bu]_{(0,T)} \leq C(\cI+[\bu]_{(0,T)}^2) 
\end{align}
holds for any fixed $T>0$
when the transformed problem \eqref{LNS} admits a unique solution $\bu$ on $(0,T)$ 
sufficiently small in the norm $[\cdot]_{(0,T)}$. 
Here, we have defined 
\begin{align} \label{def:ininorm}
\cc{I}=\sum_{i=0,2}\|\bv_0\|_{B^{2(1-1/p)}_{q_i,p}(\Omega)} 
\end{align}
and $C>0$ is a constant independent of $\bu$ and $T$. 

To prove the a priori estimate \eqref{esti:apiori}, we show 
\begin{align} \label{esti:linn}
[\bu]_{(0,T)} \leq C\cc{N}(\bff(\bu),\bg(\bu),g(\bu),\bh(\bu)\bn,\bv_0) 
\end{align}
in Subsection \ref{subsec:linesti:gen} and 
\begin{align} \label{esti:nonlin}
\cc{N}(\bff(\bu),\bg(\bu),g(\bu),\bh(\bu)\nml,\bv_0) \leq C(\cI+[\bu]_{(0,T)}^2) 
\end{align}
in Subsection \ref{subsec:nonlinesti:gen}.
Then, we obtain the global well-posedness of the transformed problem \eqref{LNS} 
and the estimate
\begin{align} \label{esti:gen}
[\bu]_{(0,T)}\leq C\eps
\end{align}
of the solution $\bu$. 
Here, we have let 
\begin{align} \label{def:NormofData}
\begin{aligned}
&\cc{N}(\bF,\bG,G,\bH,\bv_0) 
\\ &\begin{aligned}
= \sum_{i=0,2}\big(&\|<t>^b\bF\|_{L_{p}(0,T;L_{q_i}(\Omega))}
+\|(E_T[<t>^b\bG],\del_tE_T[<t>^b\bG])\|_{L_{p}(\BR;L_{q_i}(\Omega))}
\\ & +\|\tdh E_T[<t>^b(G,\bH)]\|_{L_{p}(\BR;L_{q_i}(\Omega))}
\\ & +\|E_T[<t>^b(G,\bH)]\|_{L_{p}(\BR;H^1_{q_i}(\Omega))}
+\|\bv_0\|_{B^{2(1-1/p)}_{q_i,p}(\Omega)}\big), 
\end{aligned}
\end{aligned}
\end{align}
where the extention operator $E_T$ is defined by 
\begin{align} \label{def:ET}
E_Tf(t)=\left\{ \begin{array}{ll}
f(t)&0<t<T,\\
f(2T-t)&T\leq t<2T,\\
0&\text{otherwise} 
\end{array} \right.
\end{align} 
for a function $f$ defined on $[0,T)$ with $f|_{t=0}=0$. 
Note that  
\begin{align} \label{eq:EstiET}
\quad
\|\del_t^\ell\nabla^m E_T[f]\|_{L_{p}(\BR;L_{q}(\Omega))}
\leq C\|\del_t^\ell\nabla^m f\|_{L_{p}(0,T;L_{q}(\Omega))} 
\quad (\ell=0,1,\,m\in\bb{N}\cup\{0\})
\quad 
\end{align}
for $p,q\in[1,\infty]$ by 
\begin{align} \label{eq:deltET}
\del_tE_T[f](t)=\left\{ \begin{array}{ll}
\del_tf(t)&0<t<T,\\
-\del_tf(2T-t)&T<t<2T,\\
0&\text{otherwise}. 
\end{array} \right.
\end{align}
The decay properties \eqref{eq:decay} are obtained by 
\begin{align} \label{eq:decaypf}
\|<t>^{\bmpqz{m}{\infty}{\tq_m}}\nabla^m\bu\|_{L_{\infty}(0,T;L_{\tq_m}(\Omega))}
\leq C[\bu]_{(0,T)} \leq C\eps 
\end{align}
if $m=0,1$ and $(\infty,\tq_m)$ satisfies \eqref{eq:exp_restriction}.

\subsection{Estimate for the Stokes problem in the general domain} \label{subsec:linesti:gen}

In this subsection, 
we prove the estimate \eqref{esti:linn}. 
Because $\bu$ can be regarded as the solution to the Stokes problem \eqref{S} with 
\begin{align}
(\bv_0,\bF,\bG,G,\bH)=(\bv_0,\bff(\bu),\bg(\bu),g(\bu),\bh(\bu)),
\end{align}
it suffices to prove the corresponding estimate \eqref{esti:MRS} 
in the following theorem. 
To do so, we combine the maximal regularity Theorem \ref{Thm:SMR} with 
\begin{align}
(\bG_b,G_b,\bH_b)=(E_T[<t>^b\bG],E_T[<t>^bG],E_T[<t>^b\bH])
\end{align}
and $L_q$-$L_r$ estimates \eqref{esti:LqLr} 
for the decay rate $\sigma_m(\tp,\tq)$ 
with the condition \textup{(C1)} in Theorem \ref{Thm:GWPgen}. 
\begin{Thm} \label{Thm:MR}
Let $1<p,q<\infty$ and $T\in(0,\infty]$. 
Assume that Assumptions 1 and 2 hold and that the $L_q$-$L_r$ estimates holds for the decay rate $\sigma_m(q,r)$ 
defined for $m=0,1,2$ and for $(q,r)\in(1,\infty)$ with $q\geq r$ 
and satisfying the condition \textup{(C1)} in Theorem \ref{Thm:GWPgen}. 
For any $\bv_0\in\cc{D}_{q,p}(\Omega)$ 
and right members $(\bF,\bG,G,\bH)$ defined on $(0,T)$ satisfying 
\begin{align} \label{eq:DataClassMR:gen}
\begin{aligned}
&<t>^b\bF \in L_{p}(0,T;L_{q}(\Omega)^N), \quad  
E_T[<t>^b\bG] \in H^1_{p}(\BR;L_{q}(\Omega)^N), 
\\ &E_T[<t>^bG] \in H^{1/2}_{p}(\BR;L_{q}(\Omega))\cap L_{p}(\BR;H^{1}_{q}(\Omega)), 
\\ &E_T[<t>^b\bH] \in H^{1/2}_{p}(\BR;L_{q}(\Omega)^N)\cap L_{p}(\BR;H^{1}_{q}(\Omega)^N) 
\end{aligned}
\end{align}
and the compatibility condition
\begin{align} \label{eq:DataCCMRS}
(G(t),\varphi)_\Omega = (\bG(t),\nabla\varphi)_\Omega 
\text{ for any } \varphi\in\sppresz[q'], 
\end{align}
the Stokes problem \eqref{S} admits unique solutions 
\begin{gather}
\bU \in H^{1}_{p}(0,T;L_{q}(\Omega)^N) \cap L_{p}(0,T;H^{2}_{q}(\Omega)^N), \ 
\Pres \in L_{p}(0,T;\sppres). 
\end{gather}
Moreover, the solutions possess the estimate
\begin{align} \label{esti:MRS}
[\bU]_{(0,T)} \leq C_b\cc{N}(\bF,\bG,G,\bH,\bv_0) 
\end{align}
for $b\geq0$, where $C_b>0$ is a constant independent of $T$. 
\end{Thm}

To prove Theorem \ref{Thm:MR}, 
it suffices to construct a solution to \eqref{S} with the estimate 
\begin{align} \label{esti:lin0}
\begin{aligned}
&\|<t>^{\bmmpq[2]}\del_t\bU\|_{L_{\tp_2}(0,T;L_{\tq_2}(\Omega))} 
\leq C_b\cc{N}(\bF,\bG,G,\bH,\bv_0),   \\ 
&\|<t>^{\bmmpq}\nabla^m\bU\|_{L_{\tp_m}(0,T;L_{\tq_m}(\Omega))} 
\leq C_b\cc{N}(\bF,\bG,G,\bH,\bv_0) 
\end{aligned}
\end{align}
for any $m=0,1,2$ and $(\tp_m,\tq_m)$ satisfying \eqref{eq:exp_restriction}
because the uniqueness is obtained by \cite[Theorem 3.2]{Shiba15DE}
and because \eqref{esti:MRS} can be obtained 
by \eqref{esti:lin0} and the definition \eqref{def:Normofu} of $[\bu]$. 
To this end, we consider the time-shifted Stokes system \eqref{SS} 
to deduce a sufficient decay of the solution 
and, then consider the system for the difference of the solutions to 
\eqref{SS} and the Stokes system \eqref{S}. 
We estimate the solution to the former system 
by the maximal $L_p$-$L_q$ regularity stated in Theorem \ref{Thm:SMR} 
and, to the latter, 
by the $L_q$-$L_r$ estimates \eqref{esti:LqLr} of the Stokes semigroup
for the decay rate $\sigma_m(q,r)$ with condition \textup{(C1)}.  

Divide the solutions $\bU$ and $\Pres$ of the Stokes equation \eqref{S} 
into three parts as 
\begin{align} \label{eq:DecompU}
\text{$\bU=\bU^1+\bU^2+\bU^3$ and $\Pres=\Pres^1+\Pres^2+\Pres^3$} 
\end{align}
so that each part satisfies the following equation for sufficiently large $\lambda_1$. 
\begin{align} \label{S1}
\qquad \left\{
\begin{array}{ll}
\del_t\bU^1 + \lambda_1\bU^1 - \Divv\bS(\bU^1,\Pres^1) = \bF, \quad
\divv \bU^1=G & \text{in } \Omega\times(0,\infty), \\
\bS(\bU^1,\Pres^1)\nml = \bH & \text{on } \del\Omega\times(0,\infty), \\
\bU^1|_{t=0}=\bv_0 & \text{in } \Omega, 
\end{array}
\right. \qquad \\ \label{S2}
\qquad \left\{
\begin{array}{ll}
\del_t\bU^2 + \lambda_1\bU^2 - \Divv\bS(\bU^2,\Pres^2) = \lambda_1\bU^1, \quad
\divv \bU^2=0 & \text{in } \Omega\times(0,\infty), \\
\bS(\bU^2,\Pres^2)\nml = 0 & \text{on } \del\Omega\times(0,\infty), \\
\bU^2|_{t=0}=0 & \text{in } \Omega, 
\end{array}
\right. \qquad \\ \label{S3}
\qquad \left\{
\begin{array}{ll}
\del_t\bU^3 - \Divv\bS(\bU^3,\Pres^3) = \lambda_1\bU^2, \quad
\divv \bU^3=0 & \text{in } \Omega\times(0,\infty), \\
\bS(\bU^3,\Pres^3)\nml = 0 & \text{on } \del\Omega\times(0,\infty), \\
\bU^3|_{t=0}=0 & \text{in } \Omega. 
\end{array}
\right. \qquad 
\end{align}

\begin{Rem} \label{Rem:DAq}
Note that the right-hand side of the first equation of \eqref{S3}, $\lambda_1\bU^2$, 
belongs to $D(A_q)$
while that of \eqref{S2}, $\lambda_1\bU^1$, in general does not. 
This is why we divide the solutions into three parts rather than 
two parts as in Shibata \cite[p. 448]{Shiba17CIME}, 
which is important in estimating 
\begin{align}
\nabla^2\bU^3(t)=\int_0^t \nabla^2e^{-(t-s)A_q}\lambda_1\bU^2(s) \, ds. 
\end{align}
In fact, the right-hand side will have a singularity on $s=t$
if we estimate it only by the pointwise estimate of the semigroup as  
\begin{align}
\left\| \int_0^t \nabla^2e^{-(t-s)A_q}\lambda_1\bU^2(s) \, ds \right\|_{L_{q}(\Omega)}
&\leq \int_0^t \|\nabla^2e^{-(t-s)A_q}\lambda_1\bU^2(s)\|_{L_{q}(\Omega)} \, ds
\\ &\leq C \int_0^t (t-s)^{-1} \|\bU^2(s)\|_{L_{q}(\Omega)} \, ds. 
\end{align}
We overcome this difficulty by the observations 
that $\nabla^2$ and $A_q$ are comparable 
and that we can exchange $A_q$ and $e^{-tA_q}$ thanks to $\bU^2\in D(A_q)$. 
\begin{align} \label{eq:swap}
A_qe^{-tA_q}\bU^2=e^{-tA_q}A_q\bU^2. 
\end{align}
\end{Rem}

\subti{Estimate for $\bU^3$}

First, we prove the estimate for the solution $\bU^3$ to \eqref{S3}  
for $m=0,1,2$ and $(\tp_m,\tq_m)$ with \eqref{eq:exp_restriction}. 
\begin{align} \label{esti:lin0U3}
\begin{aligned}
&\|<t>^{\bmmpq[2]}\del_t\bU^3\|_{L_{\tp_2}(0,\TTT;L_{\tq_2}(\Omega))} 
\leq C\sum_{i=0,2}\|<t>^b\bU^2\|_{L_{p}(0,T;H^2_{q_i}(\Omega))}, 
\\ &\|<t>^{\bmmpq}\nabla^m\bU^3\|_{L_{\tp_m}(0,\TTT;L_{\tq_m}(\Omega))} 
\leq C\sum_{i=0,2}\|<t>^b\bU^2\|_{L_{p}(0,T;H^2_{q_i}(\Omega))}. 
\end{aligned}
\end{align} 

Let us decompose the domains of the norms in the left-hand side 
as $(0,\TTT)=(0,2]\cup(2,T)$. Then, it suffices to show 
\begin{align} \label{esti:lin0U3L} 
\begin{aligned}
&\|<t>^{\bmmpq[2]}\del_t\bU^3\|_{L_{\tp_2}(0,2;L_{\tq_2}(\Omega))} 
\leq C\sum_{i=0,2}\|<t>^b\bU^2\|_{L_{p}(0,T;H^2_{q_i}(\Omega))}, 
\\ &\|<t>^{\bmmpq}\nabla^m\bU^3\|_{L_{\tp_m}(0,2;L_{\tq_m}(\Omega))} 
\leq C\sum_{i=0,2}\|<t>^b\bU^2\|_{L_{p}(0,T;H^2_{q_i}(\Omega))},  
\end{aligned} 
\\ \label{esti:lin0U3G} 
\begin{aligned}
&\|<t>^{\bmmpq[2]}\del_t\bU^3\|_{L_{\tp_2}(2,\TTT;L_{\tq_2}(\Omega))} 
\leq C\sum_{i=0,2}\|<t>^b\bU^2\|_{L_{p}(0,T;H^2_{q_i}(\Omega))}, 
\\ &\|<t>^{\bmmpq}\nabla^m\bU^3\|_{L_{\tp_m}(2,\TTT;L_{\tq_m}(\Omega))} 
\leq C\sum_{i=0,2}\|<t>^b\bU^2\|_{L_{p}(0,T;H^2_{q_i}(\Omega))}.  
\end{aligned}
\end{align} 

\subti{Estimate for $\bU^3$ on $(2,T)$}

We first prove the estimate \eqref{esti:lin0U3G} of $\bU^3$ on $(2,T)$. 
Initially, we prove the second inequality of \eqref{esti:lin0U3G} 
and we show the first inequality in \eqref{esti:lin0U3t} below. 
For this purpose, we decompose $\bU^3$ as 
\begin{align} \label{eq:DecompU3}
\begin{aligned}
\bU^3(t)
&=\int_0^t e^{-A_q(t-s)} \lambda_1\bU^2(s)\,ds\\
&=\left(\int_0^{t/2}+\int_{t/2}^{t-1}+\int_{t-1}^{t}\right) 
e^{-A_q(t-s)} \lambda_1\bU^2(s)\,ds 
\\ &=\bU^{31}(t)+\bU^{32}(t)+\bU^{33}(t) 
\end{aligned}
\end{align}
by setting 
\begin{align}
&\bU^{31}(t)=\int_0^{t/2} e^{-A_q(t-s)} \lambda_1\bU^2(s)\,ds, \quad  
\bU^{32}(t)=\int_{t/2}^{t-1} e^{-A_q(t-s)} \lambda_1\bU^2(s)\,ds, 
\\ &\bU^{33}(t)=\int_{t-1}^{t} e^{-A_q(t-s)} \lambda_1\bU^2(s)\,ds 
\end{align}
to use the relation 
\begin{align} \label{eq:stRelation}
\begin{array}{ll}
s<t-s\sim t & \text{when } 0<s<t/2, \ t>2, \\ 
t-s<s\sim t, \ t-s>1 & \text{when } t/2<s<t-1, \ t>2, \\ 
t-s<s\sim t, \ t-s<1 & \text{when } t-1<s<t, \ t>2. 
\end{array}
\end{align}
Then, we show the second inequality of \eqref{esti:lin0U3G} 
with $\bU^3$ replaced by $\bU^{3i}$ for each $i=1,2,3$. 

We first estimate $\bU^{33}$. 
To overcome the singlarity on $s=t$, (see Remark \ref{Rem:DAq}) 
we apply the following lemma and employ the formula $A_qe^{-tA_q}\bU^2=e^{-tA_q}A_q\bU^2$, 
which is obtained owing to $\bU^2\in D(A_q)$. 

\begin{Lem} \label{Lem:EquivDAqH2q}
Let $1<q<\infty$. 
The norms associated with $D(A_q)$ and $H^{2}_{q}(\Omega)$ are equivalent, 
that is, there exists $C>0$ satisfying 
\begin{align}
C^{-1}\|\bu\|_{H^{2}_{q}(\Omega)} \leq \|(\bu,A_q\bu)\|_{L_q(\Omega)} 
\leq C\|\bu\|_{H^2_q(\Omega)} \quad \text{for any }\bu\in D(A_q). 
\end{align}
\end{Lem}
\begin{proof}
By the definition \eqref{def:Aq} of $A_q$ and \eqref{def:sopres} of $\sopres$, 
and the estimate \eqref{esti:wDp} for $\sopresz$, 
we get $\|(\bu,A_q\bu)\|_{L_{q}(\Omega)}\leq C\|\bu\|_{H^{2}_{q}(\Omega)}$. 
To prove the other estimate, 
we set $\bff=(\lambda_0+A_q)\bu$ for sufficiently large $\lambda_0>0$. 
Then, similarly to the argument to derive the reduced Stokes equation \eqref{RSv0}, 
$\bu$ and $\pres=\sopres\bu$ satisfy
\begin{align} \label{SRlam0:Lem}
\left\{
\begin{array}{ll}
\lambda_0 \bu - \Divv \bS(\bu,\pres) = \bff, \quad \divv \bu = 0 & \text{in } \Omega \\
\bS(\bu,\pres)\nml = 0 & \text{in } \del\Omega. 
\end{array}
\right. 
\end{align}
Thus, by the uniqueness and the estimate 
\begin{align}
\|\bu\|_{H^{2}_{q}(\Omega)} \leq C\|\bff\|_{L_{q}(\Omega)}
\end{align}
for \eqref{SRlam0:Lem} due to \cite[Theorem 1.5 (1)]{Shiba13MFM} 
and by $\bff=(\lambda_0+A_q)\bu$, we obtain the desired estimate. 
\end{proof}

\subti{Estimate for $\bU^{33}$}

Now, we show the estimate 
for the third part $\bU^{33}$ of the solution formula \eqref{eq:DecompU3} of $\bU^3$:  
for $m=0,1,2$ and $(\tp_m,\tq_m)$ satisfying \eqref{eq:exp_restriction},  
\begin{align} \label{esti:lin0U33} 
\begin{aligned}
&\|<t>^{\bmmpq}\nabla^m\bU^{33}\|_{L_{\tp_m}(2,\TTT;L_{\tq_m}(\Omega))} 
\leq C\sum_{i=0,2}\|<t>^b\bU^2\|_{L_{p}(0,T;H^2_{q_i}(\Omega))}.  
\end{aligned}
\end{align} 
Because $t-s<s\sim t$ and $\bmmpq\leq b$, 
see the relation \eqref{eq:stRelation} and the definition \eqref{def:bmpq} of $\bmmpq$, 
we obtain 
\begin{align} \label{eq:estiU33a}
\begin{aligned}
&\|<t>^{\bmmpq}\nabla^m\bU^3\|_{L_{\tp_m}(2,\TTT;L_{\tq_m}(\Omega))}
\\ &=\left\| <t>^{\bmmpq} \nabla^m \int_{t-1}^t e^{-A_{{\tq_m}}(t-s)} \lambda_1\bU^2(s)\,ds 
\right\|_{{L_{\tp_m}(2,\TTT;L_{\tq_m}(\Omega))}}
\\ &\leq C\left\| \nabla^m \int_{t-1}^t e^{-A_{{\tq_m}}(t-s)} \lambda_1<s>^b\bU^2(s)\,ds 
\right\|_{{L_{\tp_m}(2,\TTT;L_{\tq_m}(\Omega))}}
\\ &\leq C\left\| \int_{t-1}^t 
\|\nabla^m e^{-A_{{\tq_m}}(t-s)} <s>^b\bU^2(s)\|_{L_{\tq_m}(\Omega)}
\,ds \right\|_{{L_{\tp_m}(2,\TTT)}}
\\ &\leq C\sum_{i=0,2}\left\| \int_{t-1}^t 
\|e^{-A_{q_i}(t-s)} <s>^b\bU^2(s)\|_{H^2_{q_i}(\Omega)}
\,ds \right\|_{{L_{\tp_m}(2,\TTT)}}, 
\end{aligned}
\end{align}
where we have used the Sobolev embedding and \eqref{eq:exp_restriction} 
in the last inequality. 
We apply Lemma \ref{Lem:EquivDAqH2q} and the formula \eqref{eq:swap} 
to estimate the right-hand side as follows.  
\begin{align}
\begin{aligned}
&\left\| \int_{t-1}^t  
\|e^{-A_{q_i}(t-s)}<s>^b\bU^2(s)\|_{H^2_{q_i}(\Omega)} \,ds \right\|_{{L_{\tp_m}(2,\TTT)}}
\\ \notag &\leq C\left\| \int_{t-1}^t  
\|(e^{-A_{q_i}(t-s)}<s>^b\bU^2(s),A_{q_i}e^{-A_{q_i}(t-s)}<s>^b\bU^2(s))\|_{L_{q_i}(\Omega)} 
\,ds \right\|_{{L_{\tp_m}(2,\TTT)}}
\\ \notag &=C\left\| \int_{t-1}^t  
\|(e^{-A_{q_i}(t-s)}<s>^b\bU^2(s),e^{-A_{q_i}(t-s)}<s>^bA_{q_i}\bU^2(s))\|_{L_{q_i}(\Omega)} 
\,ds \right\|_{{L_{\tp_m}(2,\TTT)}},
\end{aligned}
\end{align}
and this term is estimated as follows
from $\|e^{-A_qt}\|_{\cc{L}(L_{q_i}(\Omega))}\leq C$ ($0\leq t\leq 1$), 
Young's inequality, and Lemma \ref{Lem:EquivDAqH2q}. 
\begin{align} \label{eq:estiU33c}
\begin{aligned} 
&\left\| \int_{t-1}^t  
\|(e^{-A_{q_i}(t-s)}<s>^b\bU^2(s),e^{-A_{q_i}(t-s)}<s>^bA_{q_i}\bU^2(s))\|_{L_{q_i}(\Omega)} 
\,ds \right\|_{{L_{\tp_m}(2,\TTT)}}
\\ &\leq C\left\| \int_{t-1}^t  
\|(<s>^b\bU^2(s),<s>^bA_{q_i}\bU^2(s))\|_{L_{q_i}(\Omega)} \,ds \right\|_{{L_{\tp_m}(2,\TTT)}}
\\ &\leq C\left\| \int_{\BR} \fb1_{(0,1)}(t-s) 
\fb1_{(0,\TTT)}(s)\|<s>^b(\bU^2(s),A_{q_i}\bU^2(s))\|_{L_{q_i}(\Omega)}\,ds 
\right\|_{{L_{\tp_m}(2,\TTT)}}
\\ &\leq C\| \fb1_{(0,1)} \|_{L_{r}(\BR)} 
\|\fb1_{(0,\TTT)}(s)<s>^b(\bU^2(s),A_{q_i}\bU^2(s))\|_{{L_{p}(\BR;L_{q_i}(\Omega))}} 
\\ &= C\| 1 \|_{L_{r}(0,1)} 
\|<s>^b(\bU^2(s),A_{q_i}\bU^2(s))\|_{{L_{p}(0,T;L_{q_i}(\Omega))}}
\\ &\leq C\sum_{i=0,2}\|<t>^b\bU^2\|_{L_{p}(0,T;H^2_{q_i}(\Omega))}, 
\end{aligned}
\end{align}
where the exponent $r$ is defined by $1/r+1/p=1/\tp_m+1$ 
and $\fb1_A$ is the characteristic function on a set $A$. 
In this way, we obtain the estimate \eqref{esti:lin0U33}.

\subti{Estimate for $\bU^{31}$}

We next prove the estimate 
for the first part $\bU^{31}$ of the solution formula \eqref{eq:DecompU3} of $\bU^3$: 
for $m=0,1,2$ and $(\tp_m,\tq_m)$ with \eqref{eq:exp_restriction}, 
\begin{align} \label{esti:lin0U31}
\begin{aligned}
&\|<t>^{\bmmpq}\nabla^m\bU^{31}\|_{L_{\tp_m}(2,\TTT;L_{\tq_m}(\Omega))} 
\leq C\sum_{i=0,2}\|<t>^b\bU^2\|_{L_{p}(0,\TTT;H^2_{q_i}(\Omega))}. 
\end{aligned}
\end{align} 

By the assumption on the $L_q$-$L_r$ estimate for the decay rate $\sigma_m(q,r)$ 
with the conditions \textup{(C1)} in Theorem \ref{Thm:GWPgen}, 
as well as by $s\leq t-s\sim t$, see \eqref{eq:stRelation}, and H\"{o}lder's inequality,  
\begin{align} \label{eq:estiU31a}
\begin{aligned}
&\|\nabla^m\bU^{31}\|_{L_{\tq_m}(\Omega)} 
\\ &=\left\| \nabla^m \int_0^{t/2} e^{-A_q(t-s)} \lambda_1\bU^2(s)\,ds 
\right\|_{L_{\tq_m}(\Omega)} 
\\ &\leq \int_0^{t/2} \|\nabla^m e^{-A_q(t-s)} \lambda_1\bU^2(s) 
\|_{L_{\tq_m}(\Omega)} \,ds  
\\ &\leq C \int_0^{t/2} (t-s)^{-\sigma_m(q_0,\tq_m)} 
\|\bU^2(s)\|_{L_{q_0}(\Omega)} \,ds  
\\ &\leq C <t>^{-\sigma_m(q_0,\tq_m)} 
\int_0^{t/2} <s>^{-b} \|<s>^b\bU^2(s)\|_{L_{q_0}(\Omega)} \,ds 
\\ &\leq C <t>^{-\sigma_m(q_0,\tq_m)}
\|<s>^{-b}\|_{L_{p'}(0,\TTT)}\|<s>^b\bU^2(s)\|_{L_{p}(0,\TTT;L_{q_0}(\Omega))} 
\end{aligned}
\end{align}
for $t\in(2,T)$. 
By multiplying each term by $<t>^{\bmmpq}$ and taking $L_{\tp_m}(2,\TTT)$ norm, we obtain
\begin{align} \label{eq:estiU31b}
\begin{aligned}
&\|<t>^{\bmmpq}\nabla^m\bU^{31}\|_{L_{\tp_m}(2,\TTT;L_{\tq_m}(\Omega))} 
\\ &\leq C \|<t>^{\bmmpq-\sigma_m(q_0,\tq_m)}\|_{L_{\tp_m}(2,\TTT)}
\|<s>^{-b}\|_{L_{p'}(0,\TTT)}\|<s>^b\bU^2(s)\|_{L_{p}(0,T;L_{q_0}(\Omega))} 
\\ &\leq C\|<s>^b\bU^2(s)\|_{L_{p}(0,\TTT;L_{q_0}(\Omega))}
\end{aligned}
\end{align}
because 
\begin{align} 
&\|<t>^{\bmmpq-\sigma_m(q_0,\tq_m)}\|_{L_{\tp_m}(2,\TTT)}
\\ &= \left\{ \begin{aligned}
&\|<t>^{\min\{\sigma_m(q_0,\tq_m)-\frac{1}{p}-\delta,b\}-\sigma_m(q_0,\tq_m)}\|_{L_{p}(2,T)}
\leq \|<t>^{-\frac1p-\delta}\|_{L_{p}(0,\infty)} \leq C &&(\tp_m=p) 
\\ &\|<t>^{\min\{\sigma_m(q_0,\tq_m),b\}-\sigma_m(q_0,\tq_m)}\|_{L_{\infty}(2,\infty)}
\leq \|1\|_{L_{p}(0,T)}
=1 &&(\tp_m=\infty)
\end{aligned} \right.
\end{align}
by the definition \eqref{def:bmpq} 
and $\|<s>^{-b}\|_{L_{p'}(0,\TTT)}<\infty$ from $b>\frac{1}{p'}$. 
Therefore, we obtain \eqref{esti:lin0U31}.

\subti{Estimate for $\bU^{32}$}

Finally, we show the following estimate for the second part $\bU^{32}$ 
of the solution formula \eqref{eq:DecompU3} 
for $m=0,1,2$ and $(\tp_m,\tq_m)$ with \eqref{eq:exp_restriction}. 
\begin{align} \label{esti:lin0U32}
\begin{aligned}
&\|<t>^{\bmmpq}\nabla^m\bU^{32}\|_{L_{\tp_m}(2,\TTT;L_{\tq_m}(\Omega))} 
\leq C\sum_{i=0,2}\|<t>^b\bU^2\|_{L_{p}(0,T;H^2_{q_i}(\Omega))}. 
\end{aligned}
\end{align} 
This is obtained as follows. 
By the assumption on the $L_q$-$L_r$ estimate \eqref{esti:LqLr} 
for the decay rate $\sigma_m(q,r)$ with the condition \textup{(C1)} in Theorem \ref{Thm:GWPgen}, 
\begin{align} \label{eq:estiU32a}
\begin{aligned}
&\|<t>^{\bmmpq}\nabla^m\bU^{32}\|_{L_{\tp_m}(2,\TTT;L_{\tq_m}(\Omega))}
\\ &=\left\| <t>^{\bmmpq} \nabla^m \int_{t/2}^{t-1} e^{-A_q(t-s)} \lambda_1\bU^2(s)\,ds 
\right\|_{L_{\tp_m}(2,\TTT;L_{\tq_m}(\Omega))} 
\\ &\leq\left\| <t>^{\bmmpq} \int_{t/2}^{t-1} 
\| \nabla^m e^{-A_q(t-s)} \lambda_1\bU^2(s) \|_{L_{\tq_m}(\Omega)} 
\,ds \right\|_{L_{\tp_m}(2,\TTT)} 
\\ &\leq C \left\| <t>^{\bmmpq} \int_{t/2}^{t-1} (t-s)^{-\sigma_m(q_0,\tq_m)} 
\|\bU^2(s)\|_{L_{q_0}(\Omega)} \,ds \right\|_{L_{\tp_m}(2,\TTT)}. 
\end{aligned}
\end{align}
By $t-s\leq s\sim t$, (see \eqref{eq:stRelation}) we get
\begin{align} \label{eq:relationU32}
\begin{aligned}
<t>^{\bmmpq} &= <t>^{-(b-\bmmpq)}<t>^{b} \leq C(t-s)^{-(b-\bmmpq)}<s>^b. 
\end{aligned}
\end{align}
Defining $r$ by $1/r+1/p=1/\tp_m+1$ and denoting the characteristic function of a set $A$ by $\fb1_A$, 
by \eqref{eq:relationU32} and Young's inequality, we have  
\begin{align} \label{eq:estiU32b}
\begin{aligned}
&\left\| <t>^{\bmmpq} \int_{t/2}^{t-1} (t-s)^{-\sigma_m(q_0,\tq_m)} 
\|\bU^2(s)\|_{L_{q_0}(\Omega)} \,ds \right\|_{L_{\tp_m}(2,\TTT)} 
\\ &\leq C\left\| \int_{t/2}^{t-1} (t-s)^{-\sigma_m(q_0,\tq_m)+\bmmpq-b} \|<s>^b\bU^2(s)\|_{L_{q_0}(\Omega)} \,ds 
\right\|_{L_{\tp_m}(2,\TTT)} 
\\ &\leq C\left\| \int_{\BR} \fb1_{(1,\infty)}(t-s)(t-s)^{-\sigma_m(q_0,\tq_m)+\bmmpq-b} 
\fb1_{(0,T)}(s)\|<s>^b\bU^2(s)\|_{L_{q_0}(\Omega)} \,ds 
\right\|_{L_{\tp_m}(2,\TTT)}
\\ &\leq C \|s^{-b+\bmmpq-\sigma_m(q_0,\tq_m)}\|_{L_{r}(1,\infty)}
\|<s>^b\bU^2(s)\|_{L_{p}(0,\TTT;L_{q_0}(\Omega))}
\\ &\leq C\|<s>^b\bU^2(s)\|_{L_{p}(0,\TTT;L_{q_0}(\Omega))} 
\end{aligned}
\end{align}
because $\bmmpq\leq \sigma_m(q_0,\tq_m)-\frac{1}{\tp_m}-\delta$ and $b>\frac{1}{p'}$ 
yield 
\begin{align} \label{eq:exp_check_estiU32b}
-b+\bmmpq-\sigma_m(q_0,\tq_m)<-\left(1-\frac1p\right)-\frac{1}{\tp_m}=-\frac1r, 
\end{align}
which implies 
$\|s^{-b+\bmmpq-\sigma_m(q_0,\tq_m)}\|_{L_{r}(1,\infty)}<\infty$. 
Thus, the desired estimate \eqref{esti:lin0U32} is proven. 
Then, by \eqref{eq:DecompU3}, \eqref{esti:lin0U31}, \eqref{esti:lin0U32} and \eqref{esti:lin0U33}, 
we obtain the second inequality of \eqref{esti:lin0U3G} 
as the estimate for $\bU^3$ on $(2,\TTT)$.  

\subti{Estimate of $\del_t\bU^3$ on $(2,\TTT)$}

Then, we prove the first inequality of \eqref{esti:lin0U3G} 
as the estimate for $\del_t\bU^3$ on $(2,T)$ 
\begin{align} \label{esti:lin0U3t}
\|<t>^{\bmmpq[2]}\del_t\bU^3\|_{L_{\tp_2}(2,\TTT;L_{\tq_2}(\Omega))} 
\leq C\sum_{i=0,2}\|<t>^b\bU^2\|_{L_{p}(0,T;H^2_{q_i}(\Omega))}, 
\end{align}
for any $m=0,1,2$ and $(\tp_2,\tq_2)$ satisfying \eqref{eq:exp_restriction}. 
To do so, we use the fact $\bU^3$ satisfies the equation 
\begin{align}
\del_t\bU^3+A_q\bU^3=\lambda_1\bU^2, \quad \bU^3|_{t=0}=0, 
\end{align}
which is obtained by the same way we have reduced the Stokes problem \eqref{Sv0} to \eqref{Cv0}, 
and obtain 
\begin{align}
&\|<t>^{\bmmpq[2]}\del_t\bU^3\|_{L_{\tp_2}(2,\TTT;L_{\tq_2}(\Omega))} 
\\ &\leq \|<t>^{\bmmpq[2]}A_{\tq_2}\bU^3\|_{L_{\tp_2}(2,\TTT;L_{\tq_2}(\Omega))}
+ \|<t>^{\bmmpq[2]}\lambda_1\bU^2\|_{L_{\tp_2}(2,\TTT;L_{\tq_2}(\Omega))}
\\ &\leq \|<t>^{\bmmpq[2]}A_{\tq_2}\bU^3\|_{L_{\tp_2}(2,\TTT;L_{\tq_2}(\Omega))}
+ C\sum_{i=0,2}\|<t>^b\bU^2\|_{L_{p}(0,\infty;L_{q_i}(\Omega))} 
\end{align}
by \eqref{eq:exp_restriction}. 
Regarding the first term, by the solution formula \eqref{eq:DecompU3} of $\bU^3$, 
\begin{align} \label{eq:estiU3t}
\begin{aligned}
&\|<t>^{\bmmpq[2]}A_{\tq_2}\bU^3\|_{L_{\tp_2}(2,\TTT;L_{\tq_2}(\Omega))}
\\ &=\left\|<t>^{\bmmpq[2]} A_{\tq_2}\left(\int_0^{t/2}+\int_{t/2}^{t-1}+\int_{t-1}^{t}\right) 
e^{-A_q(t-s)} \lambda_1\bU^2(s)\,ds \right\|_{L_{\tp_2}(2,\TTT;L_{\tq_2}(\Omega))}
\\ &\leq \left\|<t>^{\bmmpq[2]} \int_0^{t/2} 
\| A_{\tq_2}e^{-A_{\tq_2}(t-s)} \lambda_1\bU^2(s)\|_{L_{\tq_2}(\Omega)}
\,ds \right\|_{L_{\tp_2}(2,\TTT)}
\\ &+ \left\|<t>^{\bmmpq[2]} \int_{t/2}^{t-1} 
\| A_{\tq_2}e^{-A_{\tq_2}(t-s)} \lambda_1\bU^2(s)\|_{L_{\tq_2}(\Omega)}
\,ds \right\|_{L_{\tp_2}(2,\TTT)}
\\ &+ \left\|<t>^{\bmmpq[2]} \int_{t-1}^{t} 
\| A_{\tq_2}e^{-A_{\tq_2}(t-s)} \lambda_1\bU^2(s)\|_{L_{\tq_2}(\Omega)}
\,ds \right\|_{L_{\tp_2}(2,\TTT)}. 
\end{aligned}
\end{align}
Because 
\begin{align}
\del_te^{-tA_q}\bff+A_qe^{-tA_q}\bff=0 \quad (\bff\in J_q(\Omega)), 
\end{align}
we write 
$A_{\tq_2}e^{-A_{\tq_2}(t-s)} \lambda_1\bU^2(s)=-\del_te^{-A_{\tq_2}(t-s)} \lambda_1\bU^2(s)$
and use the $L_q$-$L_r$ estimate \eqref{esti:LqLr}. 
Then, the first term and second term of the right-hand side of \eqref{eq:estiU3t} 
are estimated by 
\begin{align}
&C\left\|<t>^{\bmmpq[2]} \int_0^{t/2} 
(t-s)^{-\sigma_2(q_0,\tq_2)} \| \lambda_1\bU^2(s)\|_{L_{\tq_2}(\Omega)}
\,ds \right\|_{L_{\tp_2}(2,\TTT)}, 
\\ & C\left\|<t>^{\bmmpq[2]} \int_{t/2}^{t-1} 
(t-s)^{-\sigma_2(q_0,\tq_2)} \| \lambda_1\bU^2(s)\|_{L_{\tq_2}(\Omega)}
\,ds \right\|_{L_{\tp_2}(2,\TTT)}, 
\end{align}
respectively. 
We continue the estimate the first term of the right-hand side of \eqref{eq:estiU3t} 
in the same way as in \eqref{eq:estiU31a} and \eqref{eq:estiU31b}, 
the second term as in \eqref{eq:estiU32b},  
and the third term as in \eqref{eq:swap} and \eqref{eq:estiU33c}. 
Then, we obtain the desired estimate for $\del_t\bU^3$, 
and summarizing the arguments above yields the estimate \eqref{esti:lin0U3G} 
for $\bU^3$ on $(2,\TTT)$.

\subti{Estimate for $\bU^{3}$ on $(0,2)$}

We now show the estimate \eqref{esti:lin0U3L} on $(0,2)$ for $\bU^3$, 
which is defined as the solution of \eqref{S3}. 
We apply the maximal regularity locally in time, 
which is proven due to Shibata \cite[Theorem 3.2]{Shiba15DE}. 
We use this for the case $(\bG,G,\bH,\bv_0)=(\fb0,0,\fb0,\fb0)$: 
\begin{Thm}[\cite{Shiba15DE}] \label{Thm:MRL}
Let $1<p,q<\infty$ and $T>0$.  
Under Assumptions 1 and 2, for any 
\begin{align} \label{eq:DataClassMRL}
\begin{gathered}
\bF \in L_{p}(0,T;L_{q}(\Omega)^N)
\end{gathered}
\end{align}
the Stokes problem \eqref{S} with $(\bG,G,\bH,\bv_0)=(\fb0,0,\fb0,\fb0)$
admits unique solutions 
\begin{gather}
\bU \in H^{1}_{p}(0,T;L_{q}(\Omega)^N) \cap L_{p}(0,T;H^{2}_{q}(\Omega)^N), \ 
\Pres \in L_{p}(0,T;\sppres). 
\end{gather}
Moreover, the solutions possess the following estimate for some $\gamma_0>0$: 
\begin{align} \label{esti:MRL}
\begin{aligned}
\|\del_t\bU\|_{L_{p}(0,T;L_{q}(\Omega))}+\|\bU\|_{L_{p}(0,T;H^2_{q}(\Omega))}
+\|\nabla\Pres\|_{L_{p}(0,T;L_{q}(\Omega))} 
\leq Ce^{\gamma_0T} \|\bF\|_{L_{p}(0,T;L_{q}(\Omega))}. 
\end{aligned}
\end{align}
\end{Thm}

Also, we use the following embedding estimate 
for any $m=0,1,2$, $(\tp_m,\tq_m)$ satisfying \eqref{eq:exp_restriction} 
and $\bu\in \bigcap_{i=0,2} (H^{1}_{p}(0,T;L_{q_i}(\Omega))\cap L_{p}(0,T;H^{2}_{q_i}(\Omega)))$, 
\begin{align} \label{esti:embedpmqm}
\begin{aligned}
&\|\nabla^m\bu\|_{L_{\tp_m}(0,T;L_{\tq_m}(\Omega))} 
\\ &\leq \sum_{i=0,2}(\|\del_t\bu\|_{L_{p}(0,T;L_{q_i}(\Omega))}
+\|\bu\|_{L_{p}(0,T;H^2_{q_i}(\Omega))}+\|\bu|_{t=0}\|_{B^{2(1-1/p)}_{q_i,p}(\Omega)}). 
\end{aligned}
\end{align}
To prove this, consider the following cases. 
\begin{enumerate}
\renewcommand{\labelenumi}{(\roman{enumi})}
\setlength{\parskip}{0mm} 
\setlength{\itemsep}{0mm} 
\item $\tp_m=p$ and $\tq_m\in[q_0,q_2]$ \ ($m=0,1,2$)
\item $\tp_m=p$ and $\tq_m\in(q_2,\infty]$ \ ($m=0,1$ by \eqref{eq:exp_restriction})
\item $\tp_m=\infty$ and $\tq_m\in[q_0,q_2]$ \ ($m=0,1$ by \eqref{eq:exp_restriction})
\item $\tp_m=\infty$ and $\tq_m\in(q_2,\infty]$ \ ($m=0$ by \eqref{eq:exp_restriction})
\end{enumerate}
The case (i) is clear. The case (ii) is obtained by the Sobolev embedding 
\begin{align}
&\|\nabla^m\bu\|_{L_{p}(0,T;L_{\tq_m}(\Omega))} 
\leq C\|\bu\|_{L_{p}(0,T;H^{m+1}_{q_2}(\Omega))}
\\ &\leq C\sum_{i=0,2}(\|\del_t\bu\|_{L_{p}(0,T;L_{q_i}(\Omega))}
+\|\bu\|_{L_{p}(0,T;H^2_{q_i}(\Omega))}
+\|\bu|_{t=0}\|_{B^{2(1-1/p)}_{q_i,p}(\Omega)}). 
\end{align}
To show \eqref{esti:embedpmqm} for the case (iii) and (iv), we use the embedding 
\begin{align} \label{eq:embedBUC}
H^1_p(0,\infty;X_0)\cap L_p(0,\infty;X_1) \subset C_b([0,\infty);(X_0,X_1)_{1-1/p,p}) 
\end{align}
for two Banach spaces $X_0$ and $X_1$, 
where $C_b(I;X)$ stands for the space of the $X$-valued bounded continuous functions on $I$.  
(see e.g. \cite[Corollary 1.14]{Luna18SNS}.) 
We also set 
\begin{align}
\bu^1=e^{-tA_q}\bu|_{t=0}
\end{align}
so that $(\bu-\bu^1)|_{t=0}=0$, $\bu=E_T[\bu-\bu^1]+\bu^1$ on $(0,T)$ and 
\begin{align} \label{eq:Esti_sg_inivel}
\|\bu^1\|_{H^1_{p}(0,\infty;L_{q}(\Omega))}+\|\bu^1\|_{L_{p}(0,\infty;H^2_{q}(\Omega))}
\leq C\|\bu|_{t=0}\|_{B^{2(1-1/p)}_{p,q}(\Omega)}
\end{align}
for $q\in(1,\infty)$, 
where $E_T$ is the extention function defined by \eqref{def:ET}. 
Because $p>2$ implies $B^{2(1-1/p)}_{\tq_m,p}(\Omega)\subset H^{1}_{\tq_m}(\Omega)$, 
by the embedding \eqref{eq:embedBUC} with $(X_0,X_1)=(L_{\tq_m}(\Omega),H^{2}_{\tq_m}(\Omega))$, 
the estimate \eqref{eq:EstiET} of $E_T$ and \eqref{eq:Esti_sg_inivel}, 
for the case (iii), we have  
\begin{align}
&\|\nabla^m\bu\|_{L_{\tp_m}(0,T;L_{\tq_m}(\Omega))} 
= \|\nabla^m(E_T[\bu-\bu^1]+\bu^1)\|_{L_{\infty}(0,T;L_{\tq_m}(\Omega))}
\\ &\leq C(\|E_T[\bu-\bu^1]+\bu^1\|_{H^1_{p}(0,\infty;L_{\tq_m}(\Omega))}
+\|E_T[\bu-\bu^1]+\bu^1\|_{L_{p}(0,\infty;H^2_{\tq_m}(\Omega))})
\\ &+\|(E_T[\bu-\bu^1]+\bu^1)|_{t=0}\|_{B^{2(1-1/p)}_{\tq_m,p}(\Omega))})
\\ &\leq C(\|(\bu-\bu^1,\bu^1)\|_{H^1_{p}(0,T;L_{\tq_m}(\Omega))}
+\|(\bu-\bu^1,\bu^1)\|_{L_{p}(0,T;H^2_{\tq_m}(\Omega))})
\\ &+\|\bu^1|_{t=0}\|_{B^{2(1-1/p)}_{\tq_m,p}(\Omega))})
\\ &\leq C\sum_{i=0,2}(\|\del_t\bu\|_{L_{p}(0,T;L_{q_i}(\Omega))}
+\|\bu\|_{L_{p}(0,T;H^2_{q_i}(\Omega))}
+\|\bu|_{t=0}\|_{B^{2(1-1/p)}_{q_i,p}(\Omega)}). 
\end{align}
For the case (iv), by the Sobolev embedding 
and by the result for the case (iii), 
\begin{align}
&\|\bu\|_{L_{\tp_0}(0,T;L_{\tq_0}(\Omega))} 
\leq C\|\bu\|_{L_{\infty}(0,T;H^1_{q_2}(\Omega))}
\\ &\leq C\sum_{i=0,2}(\|\del_t\bu\|_{L_{p}(0,T;L_{q_i}(\Omega))}
+\|\bu\|_{L_{p}(0,T;H^2_{q_i}(\Omega))}
+\|\bu|_{t=0}\|_{B^{2(1-1/p)}_{q_i,p}(\Omega)}). 
\end{align}
To summarize, \eqref{esti:embedpmqm} holds.  

Then, thanks to \eqref{esti:embedpmqm}, $\bU^3|_{t=0}=\fb0$ and Theorem \ref{Thm:MRL}, 
we obtain the estimate \eqref{esti:lin0U3L} as follows. 
\begin{align}
\begin{aligned}
&\|<t>^{\bmmpq[2]}\del_t\bU^3\|_{L_{\tp_2}(0,2;L_{\tq_2}(\Omega))} 
+\sum_{m=0,1,2} \|<t>^{\bmmpq}\nabla^m\bU^3\|_{L_{\tp_m}(0,2;L_{\tq_m}(\Omega))}
\\ &\leq <2>^{\bmmpq[2]}\|\del_t\bU^3\|_{L_{\tp_2}(0,2;L_{\tq_2}(\Omega))}
+ \sum_{m=0,1,2} <2>^{\bmmpq}\|\nabla^m\bU^3\|_{L_{\tp_m}(0,2;L_{\tq_m}(\Omega))}
\\ &\leq C\sum_{i=0,2}(\|\del_t\bU^3\|_{L_{p}(0,2;L_{q_i}(\Omega))}
+\|\bU^3\|_{L_{p}(0,2;H^2_{q_i}(\Omega))})
\\ &\leq Ce^{2\gamma_0}\sum_{i=0,2}\|\bU^2\|_{L_{p}(0,2;L_{q_i}(\Omega))}
\leq C\sum_{i=0,2}\|<t>^b\bU^2\|_{L_{p}(0,T;H^2_{q_i}(\Omega))}. 
\end{aligned}
\end{align}
Combining this with \eqref{esti:lin0U3G}, 
we obtain the estimate \eqref{esti:lin0U3} for $\bU^3$.

\subti{Estimate for $\bU^1$ and $\bU^2$}

Finally, we prove the estimate for $\bU^1$ and $\bU^2$, 
which are defined of the solution to \eqref{S1} and \eqref{S2}, respectively. 
That is, for $m=0,1,2$ and $(\tp_m,\tq_m)$ with \eqref{eq:exp_restriction}, 
\begin{gather} \label{esti:lin0U1}
\begin{aligned}
\|<t>^{\bmmpq[2]}\del_t\bU^1\|_{L_{\tp_2}(0,T;L_{\tq_2}(\Omega))} 
&\leq C\cc{N}(\bF,\bG,G,\bH,\bv_0),   \\ 
\|<t>^{\bmmpq}\nabla^m\bU^1\|_{L_{\tp_m}(0,T;L_{\tq_m}(\Omega))} 
&\leq C\cc{N}(\bF,\bG,G,\bH,\bv_0), 
\end{aligned} 
\\ \label{esti:lin0U2}
\begin{aligned}
\|<t>^{\bmmpq[2]}\del_t\bU^2\|_{L_{\tp_2}(0,T;L_{\tq_2}(\Omega))} 
&\leq C\sum_{i=0,2}\|<t>^b\bU^1\|_{L_{p}(0,T;L_{q_i}(\Omega))},   \\ 
\|<t>^{\bmmpq}\nabla^m\bU^2\|_{L_{\tp_m}(0,T;L_{\tq_m}(\Omega))} 
&\leq C\sum_{i=0,2}\|<t>^b\bU^1\|_{L_{p}(0,T;L_{q_i}(\Omega))}. 
\end{aligned}
\end{gather}
Let $k=1,2$. By $\bmmpq\leq b$ and \eqref{eq:exp_restriction}, 
\begin{align}
&\|<t>^{\bmmpq[2]}\del_t\bU^k\|_{L_{\tp_2}(0,T;L_{\tq_2}(\Omega))}
\leq \|<t>^{b}\del_t\bU^k\|_{L_{p}(0,T;L_{\tq_2}(\Omega))}.
\end{align}
Also, noting that 
\begin{align}
\del_t(<t>^b\bU^k(t))=bt<t>^{b-2}\bU^k(t)+<t>^b\del_t\bU^k(t),  
\end{align}
by $\bmmpq\leq b$ and \eqref{esti:embedpmqm}, we obtain
\begin{align}
&\|<t>^{\bmmpq}\nabla^m\bU^k\|_{L_{\tp_m}(0,T;L_{\tq_m}(\Omega))}
\leq \|<t>^{b}\nabla^m\bU^k\|_{L_{\tp_m}(0,T;L_{\tq_m}(\Omega))}
\\ \notag 
&\leq C\sum_{i=0,2}(\|\del_t(<t>^b\bU^k)\|_{L_{p}(0,T;L_{q_i}(\Omega))}
+\|<t>^b\bU^k\|_{L_{p}(0,T;H^2_{q_i}(\Omega))}
+\|\bU^k|_{t=0}\|_{B^{2(1-1/p)}_{q_i,p}(\Omega)})
\\ \notag 
&\leq C\sum_{i=0,2}(\|<t>^b\del_t\bU^k\|_{L_{p}(0,T;L_{q_i}(\Omega))}
+\|<t>^b\bU^k\|_{L_{p}(0,T;H^2_{q_i}(\Omega))}
+\|\bU^k|_{t=0}\|_{B^{2(1-1/p)}_{q_i,p}(\Omega)}). 
\end{align}
Thus, 
\begin{align*}
&\|<t>^{\bmmpq[2]}\del_t\bU^k\|_{L_{\tp_2}(0,T;L_{\tq_2}(\Omega))}
+\sum_{m=0,1,2}\|<t>^{\bmmpq}\nabla^m\bU^k\|_{L_{\tp_m}(0,T;L_{\tq_m}(\Omega))}
\\ &\leq C\sum_{i=0,2}(\|<t>^b\del_t\bU^k\|_{L_{p}(0,T;L_{q_i}(\Omega))}
+\|<t>^b\bU^k\|_{L_{p}(0,T;H^2_{q_i}(\Omega))}
+\|\bU^k|_{t=0}\|_{B^{2(1-1/p)}_{q_i,p}(\Omega)})
\end{align*}
and, by the maximal regularity stated in Theorem \ref{Thm:SMR} 
and by $(\bU^1,\bU^2)|_{t=0}=(\bv_0,\fb0)$,  
the right-hand side is estimated as follows. 
\begin{align} \label{esti:lin0U1U2}
\begin{aligned}
\sum_{i=0,2}(\|<t>^b\del_t\bU^1\|_{L_{p}(0,T;L_{q_i}(\Omega))}
+\|<t>^b\bU^1\|_{L_{p}(0,T;H^2_{q_i}(\Omega))}+\|\bU^1|_{t=0}\|_{B^{2(1-1/p)}_{q_i,p}(\Omega)}))
\\\leq C\cc{N}(\bF,\bG,G,\bH,\bv_0), 
\\ \sum_{i=0,2}(\|<t>^b\del_t\bU^2\|_{L_{p}(0,T;L_{q_i}(\Omega))}
+\|<t>^b\bU^2\|_{L_{p}(0,T;H^2_{q_i}(\Omega))}+\|\bU^2|_{t=0}\|_{B^{2(1-1/p)}_{q_i,p}(\Omega)}))
\\\leq C\cc{N}(\lambda_1\bU^1,\fb0,0,\fb0,\fb0)
=\lambda_1\sum_{i=0,2}\|<t>^b\bU^1\|_{L_{p}(0,T;L_{q_i}(\Omega))}. 
\end{aligned}
\end{align}
Hence, we obtain \eqref{esti:lin0U1} and \eqref{esti:lin0U2}. 

Now, we can conclude \eqref{esti:lin0} as follows. 
By \eqref{eq:DecompU}, \eqref{esti:lin0U1}, \eqref{esti:lin0U2} and \eqref{esti:lin0U3}, 
\begin{align}
&\|<t>^{\bmmpq[2]}\del_t\bU\|_{L_{\tp_2}(0,\TTT;L_{\tq_2}(\Omega))}
+\sum_{m=0,1,2}\|<t>^{\bmmpq}\nabla^m\bU\|_{L_{\tp_m}(0,\TTT;L_{\tq_m}(\Omega))}
\\ \notag &\leq \sum_{k=1,2,3} (\|<t>^{\bmmpq[2]}\del_t\bU^k\|_{L_{\tp_2}(0,\TTT;L_{\tq_2}(\Omega))}
+\sum_{m=0,1,2}\|<t>^{\bmmpq}\nabla^m\bU^k\|_{L_{\tp_m}(0,\TTT;L_{\tq_m}(\Omega))})
\\ &\leq C\bigg(\cc{N}(\bF,\bG,G,\bH,\bv_0) 
+\sum_{i=0,2}(\|<t>^b\bU^1\|_{L_{p}(0,T;L_{q_i}(\Omega))}
+\|<t>^b\bU^2\|_{L_{p}(0,T;H^2_{q_i}(\Omega))}) \bigg)
\end{align}
and the second term and the third term of the right-hand side 
are estimated by $C\cc{N}(\bF,\bG,G,\bH,\bv_0)$ from \eqref{esti:lin0U1U2}.

\subsection{Estimate for the nonlinear terms in general domain} \label{subsec:nonlinesti:gen}
In this subsection, we prove \eqref{esti:nonlin}. 

We begin with the estimate of $\bu$ itself and the term $\int_0^t \nabla\bu(s)\,ds$. 
\begin{Lem} \label{Lem:nonlinesti0} 
Let $q_{03}$ and $q_{04}$ the exponents 
given in the condition \textup{(C2)} in Theorem \ref{Thm:GWPgen}. 
\begin{enumerate}
\renewcommand{\labelenumi}{(\alph{enumi})}
\setlength{\parskip}{0mm} 
\setlength{\itemsep}{0mm} 
\item The expression
\begin{align}
\|\nabla^m\bu\|_{L_{\infty}(0,T;L_{\tq_m}(\Omega))} \leq [\bu]_{(0,T)} 
\end{align}
holds if $m=0,1$ and $(\infty,\tq_m)$ satisfy \eqref{eq:exp_restriction}. 
\item The statement 
\begin{align}
&\|<t>^b\del_t\bu\|_{L_{\tp_2}(0,T;L_{\tq_2}(\Omega))} \leq C[\bu]_{(0,T)}, 
\\ &\|<t>^b\nabla^m\bu\|_{L_{\tp_m}(0,T;L_{\tq_m}(\Omega))} \leq C[\bu]_{(0,T)} 
\end{align}
holds for $m=0,1,2$ and $(\tp_m,\tq_m)\in I_m$ 
satisfying \eqref{eq:exp_restriction} and $\tq_m\geq q_{04}$. 
\item There holds 
\begin{align}
\left\| \int_0^t \nabla^m \bu(s)\,ds \right\|_{L_{\infty}(0,T;L_{\tq_m}(\Omega))} 
\leq C[\bu]_{(0,T)} 
\end{align}
if $m=1,2$ and $(p,\tq_m)\in I_m$ 
satisfy \eqref{eq:exp_restriction} and $\tq_m\geq\min\{q_{04},q_{03}\}$. 
\item There holds $\|\bW(\int_0^t\nabla\bu\,ds)\|_{L_{\infty}(0,T;L_{\infty}(\Omega))}\leq C$ 
for any polynomial $\bW$. 
\end{enumerate}
\end{Lem}

\begin{proof}
(a) Because $\sigma_m(q_0,\tq_m)\geq 0$ and $b>1/p'>0$ implies 
\begin{align}
\bmpqz{m}{\infty}{\tq_m} = \min\{\sigma_m(q_0,\tq_m),b\}\geq 0, 
\end{align}
by the definition \eqref{def:Normofu} of $[\bu]_{(0,T)}$, we obtain
\begin{align}
\|\nabla^m\bu\|_{L_{\infty}(0,T;L_{\tq_m}(\Omega))}
\leq \|<t>^{\bmpqz{m}{\infty}{\tq_m}}\nabla^m\bu\|_{L_{\infty}(0,T;L_{\tq_m}(\Omega))} \leq [\bu]_{(0,T)}. 
\end{align}

(b) By the conditions \text{(C1)} and \text{(C2)} in Theorem \ref{Thm:GWPgen}, 
\eqref{eq:exp_restriction} and $\tq_m\geq q_{04}$
\begin{align} \label{eq:bleqbmpq}
\begin{aligned}
\bmpqz{m}{p}{\tq_m}
&=\min\{\sigma_m(q_0,\tq_m)-\frac1{p}-\delta,b\} 
\geq \min\{\sigma_0(q_0,q_{04})-\frac1p-\delta,b\} 
\\ &\geq \min\{(b+\frac1p)-\frac1p,b\} = b \qquad (\text{if } \tp_m=p), 
\\ \bmpqz{m}{\infty}{\tq_m}
&=\min\{\sigma_m(q_0,\tq_m),b\} 
\geq \min\{\sigma_0(q_0,q_{04}),b\} 
\\ &\geq \min\{b+\frac1p,b\} \geq b \qquad (\text{if } \tp_m=\infty) 
\end{aligned}
\end{align}
and, so, by the definition \eqref{def:Normofu} of $[\bu]_{(0,T)}$, we obtain the desired estimate. 

(c) By the conditions \text{(C1)} and \text{(C2)} in Theorem \ref{Thm:GWPgen}, 
\eqref{eq:exp_restriction} and $\tq_m\geq\min\{q_{03},q_{04}\}$, 
\begin{align}
-\bmpqz{m}{p}{\tq_m} &\leq -\sigma_m(q_0,\tq_m)+\frac1{p}+\delta 
\leq -\sigma_1(q_0,q_{03})+\frac1p+\delta 
\\ &< -1+\frac1p = -\frac1{p'} \qquad (\text{if } \tp_m=p), 
\\ -\bmpqz{m}{\infty}{\tq_m} &\leq \sigma_m(q_0,\tq_m) 
\leq -\sigma_1(q_0,q_{03}) 
\\ &< -1 < -\frac1{p'} \qquad (\text{if } \tp_m=\infty)
\end{align}
when $\tq_m\geq q_{03}$ and, by \eqref{eq:bleqbmpq} and $b>1/p'$, 
\begin{align}
-\bmmpq \leq -b < -1/p'   
\end{align}
when $\tq_m\geq q_{04}$. This implies $\|<s>^{-\bmmpq}\|_{L_{p'}(0,\infty)}\leq C$
and thus, by H\"{o}lder's inequality, we have 
\begin{align}
&\left\| \nabla^m \int_0^t \bu(s)\,ds \right\|_{L_{\infty}(0,T;L_{\tq_m}(\Omega))} 
\leq \int_0^T \|\nabla^m \bu(s)\|_{L_{\tq_m}(\Omega)} \,ds  
\\ &\quad = \int_0^T <s>^{-\bmpqz{m}{p}{\tq_m}}
\|<s>^{\bmpqz{m}{p}{\tq_m}}\nabla^m\bu(s)\|_{L_{\tq_m}(\Omega)}\,ds
\\ &\quad \leq \|<s>^{-\bmpqz{m}{p}{\tq_m}}\|_{L_{p'}(0,\infty)}
\|<s>^{\bmpqz{m}{p}{\tq_m}}\nabla^m\bu(s)\|_{L_{p}(0,T;L_{\tq_m}(\Omega))} 
\\ &\leq C[\bu]_{(0,T)}. 
\end{align}

(d) This property is immediately obtained from (c) with $m=1$, $\tq_1=\infty$. 
\end{proof}
The main step to estimate the nonlinear terms 
is to take the exponents $p_1,p_2,r_1$ and $r_2$ 
to apply Lemma \ref{Lem:nonlinesti0}
and H\"{o}lder's inequality 
\begin{align} \label{eq:Holder}
\|fg\|_{L_{p}(0,T;L_{r}(\Omega))} 
\leq \|f\|_{L_{p_1}(0,T;L_{r_1}(\Omega))} \|g\|_{L_{p_2}(0,T;L_{q_2}(\Omega))}
\quad \left( \frac{1}{p}=\frac{1}{p_1}+\frac{1}{p_2},\, \frac{1}{r}=\frac{1}{r_1}+\frac{1}{r_2} \right).
\end{align}

\subti{Estimate for $\bff(\bu)$}

By the definition \eqref{def:nonlin} of the nonlinear terms 
and Lemma \ref{Lem:nonlinesti0} (d), 
for $i=0,2$, 
\begin{align} \label{eq:estifu}
\begin{aligned}
\|<t>^b\bff(\bu)\|_{L_{p}(0,T;L_{q_i}(\Omega))}
&\leq C\left\|<t>^b\int_0^t\nabla\bu\,ds \otimes (\del_t\bu,\nabla^2\bu)\right\|_{L_{p}(0,T;L_{q_i}(\Omega))}
\\ &+ C\left\|<t>^b\int_0^t\nabla^2\bu\,ds \otimes \nabla\bu\right\|_{L_{p}(0,T;L_{q_i}(\Omega))}. 
\end{aligned}
\end{align}

The first term is estimated by $C[\bu]_{(0,T)}^2$ by Lemma \ref{Lem:nonlinesti0} as follows. 
When $i=0$, by H\"{o}lder's inequality \eqref{eq:Holder} and Lemma \ref{Lem:nonlinesti0}, we get, 
\begin{align} \label{eq:estifua}
\begin{aligned}
&\left\|<t>^b\int_0^t\nabla\bu\,ds \otimes (\del_t\bu,\nabla^2\bu)\right\|_{L_{p}(0,T;L_{q_0}(\Omega))}
\\ &\leq \left\|\int_0^t\nabla\bu\,ds\right\|_{L_{\infty}(0,T;L_{q_{03}}(\Omega))}
\left\|<t>^b(\del_t\bu,\nabla^2\bu)\right\|_{L_{p}(0,T;L_{q_{04}}(\Omega))} 
\\ &\leq C[\bu]_{(0,T)}^2. 
\end{aligned}
\end{align}
The estimate with $i=2$ is obtained 
by replacing $q_{03}$ and $q_{04}$ respectively with $\infty$ and $q_2$ in \eqref{eq:estifua} as follows. 
Bby H\"{o}lder's inequality \eqref{eq:Holder} and Lemma \ref{Lem:nonlinesti0}, 
\begin{align} \label{eq:estifuc}
\begin{aligned}
&\left\|<t>^b\int_0^t\nabla\bu\,ds \otimes (\del_t\bu,\nabla^2\bu)\right\|_{L_{p}(0,T;L_{q_2}(\Omega))}
\\ &\leq \left\|\int_0^t\nabla\bu\,ds\right\|_{L_{\infty}(0,T;L_{\infty}(\Omega))}
\left\|<t>^b(\del_t\bu,\nabla^2\bu)\right\|_{L_{p}(0,T;L_{q_{2}}(\Omega))} 
\\ &\leq C[\bu]_{(0,T)}^2. 
\end{aligned}
\end{align}

Next, we estimate the second term of the right-hand side in \eqref{eq:estifu}. 
When $i=0$, 
H\"{o}lder's inequality \eqref{eq:Holder} and Lemma \ref{Lem:nonlinesti0} yield 
\begin{align} \label{esti:ddidu}
\begin{aligned}
&\left\|<t>^b\int_0^t\nabla^2\bu\,ds \otimes \nabla\bu\right\|_{L_{p}(0,T;L_{q_0}(\Omega))}
\\ &\leq \left\|\int_0^t\nabla^2\bu\,ds\right\|_{L_{\infty}(0,T;L_{q_{03}}(\Omega))}
\left\|<t>^b\nabla\bu\right\|_{L_{p}(0,T;L_{q_{04}}(\Omega))} 
\\ &\leq C[\bu]_{(0,T)}^2. 
\end{aligned}
\end{align}
The case $i=2$ is proved by replacing $q_{03}$ and $q_{04}$ with $q_2$ and $\infty$ 
in \eqref{esti:ddidu}, respectively.  
Thus, we conclude 
$\sum_{i=0,2}\|\bff(\bu)\|_{L_{p}(0,T;L_{q_i}(\Omega))}\leq C[\bu]_{(0,T)}$.

\subti{Estimate for $\bg(\bu)$}

By the estimate \eqref{eq:EstiET} for $E_T$, 
the definition \eqref{def:nonlin} of the nonlinear terms 
and Lemma \ref{Lem:nonlinesti0} (d), for $i=0,2$, 
\begin{align} \label{eq:estibgu}
\begin{aligned}
\|E_T[<t>^b\bg(\bu)]\|_{L_{p}(\BR;L_{q_i}(\Omega))}
\leq C\left\|<t>^b\int_0^t\nabla\bu\,ds \otimes \bu \right\|_{L_{p}(0,T;L_{q_i}(\Omega))}. 
\end{aligned}
\end{align}
When $i=0$, by H\"{o}lder's inequality \eqref{eq:Holder} and Lemma \ref{Lem:nonlinesti0}, the right-hand side can be estimated as 
\begin{align} \label{esti:diu}
&\left\|<t>^b\int_0^t\nabla\bu\,ds \otimes \bu \right\|_{L_{p}(0,T;L_{q_0}(\Omega))}
\\ &\leq \left\|\int_0^t\nabla\bu\,ds\right\|_{L_{\infty}(0,T;L_{q_{03}}(\Omega))}
\left\|<t>^b\bu\right\|_{L_{p}(0,T;L_{q_{04}}(\Omega))} 
\\ &\leq C[\bu]_{(0,T)}^2. 
\end{align}
The estimate for $i=2$ can be obtained 
by replacing $q_{03}$ and $q_{04}$ respectively with $q_2$ and $\infty$.  

\begin{Rem} \label{Rem:C2}
We must assume 
\begin{align} \label{eq:ass:Rem}
\sigma_1(q_0,q_{03})>1, \ \sigma_0(q_0,q_{04})>b+1/p>1
\end{align}
for some $q_{03}$ and $q_{04}$ with $1/q_0=1/q_{03}+1/q_{04}$ to estimate 
$\|<t>^b\bg(\bu)\|_{L_{p}(0,T;L_{q_0}(\Omega))}$ independently of $T$
at least simply by applying H\"{o}lder's inequality.  
In fact, by H\"{o}lder's inequality, 
the right-hand side of \eqref{eq:estibgu} is estimated by 
\begin{align}
\left\|<t>^b\int_0^t \|\nabla\bu\|_{L_{r_1}(\Omega)}\,ds \|\bu(t)\|_{L_{r_2}(\Omega)} 
\right\|_{L_{p}(0,T)} 
\end{align}
for some $r_1,r_2\in[q_0,\infty]$ with $1/q_0=1/r_1+1/r_2$. 
To estimate this term, we must show 
$\|<t>^b\|\bu\|_{L_{r_2}(\Omega)}\|_{L_{p}(0,T)}\leq C$ 
since $\int_0^t \|\nabla\bu\|_{L_{r_1}(\Omega)}\,ds$ does not decay as $t\to\infty$. 
To obtain this result and $\int_0^t \|\nabla\bu\|_{L_{r_1}(\Omega)}\,ds\leq C$, roughly, we need 
\begin{align} \label{eq:neededdecay:Rem}
\begin{aligned}
&\|\nabla\bu(t)\|_{L_{r_1}(\Omega)} = o(t^{-1}), 
\\ &(<t>^b\|\bu(t)\|_{L_{r_2}(\Omega)})^p=o(t^{-1}) \quad \text{i.e.} 
\quad \|\bu(t)\|_{L_{r_2}(\Omega)}=o(t^{-(b+1/p)})
\end{aligned}
\end{align}
as $t\to\infty$. 
The same decay is needed for the solution $\bU$ of \eqref{S} but, 
even if $(\bF,\bG,G,\bH)=(\fb0,\fb0,0,\fb0)$ i.e., $\bU=e^{-tA_q}\bv_0$, 
we only have
\begin{align} \label{eq:gotdecay:Rem}
\begin{aligned}
\|\nabla\bU(t)\|_{L_{r_1}(\Omega)} = o(t^{-\sigma_1(q_0,r_1)}), \quad 
\|\bU(t)\|_{L_{r_2}(\Omega)} = o(t^{-\sigma_0(q_0,r_2)}). 
\end{aligned}
\end{align}
Also, because we deal with the case $(\bF,\bG,G,\bH)\ne(\fb0,\fb0,0,\fb0)$ 
from the maximal regularity for the time-shifted Stokes problem \eqref{SS}, (see Theorem \ref{Thm:SMR},) 
one can only expect that 
$\|<t>^b\nabla\bU\|_{L_{p}(0,T;L_{r_1}(\Omega))}\leq C$ 
and $\|<t>^b\bU\|_{L_{p}(0,T;L_{r_2}(\Omega))}\leq C$
which roughly means, 
\begin{align}
&(<t>^b\|\nabla\bU(t)\|_{L_{r_1}(\Omega)})^p=o(t^{-1}) \quad \text{i.e.} \quad 
\|\bU(t)\|_{L_{r_1}(\Omega)}=o(t^{-(b+1/p)}), 
\\ &(<t>^b\|\bU(t)\|_{L_{r_2}(\Omega)})^p=o(t^{-1}) \quad \text{i.e.} \quad 
\|\bU(t)\|_{L_{r_2}(\Omega)}=o(t^{-(b+1/p)}). 
\end{align}
Comparing this and \eqref{eq:gotdecay:Rem} to \eqref{eq:neededdecay:Rem}, 
we find that \eqref{eq:ass:Rem} is a necessary condition. 
This is why the slowest decay term $\bg(\bu)$ forces us 
to take $N\geq4$ even if the decay rate is as fast as that in the half-space, 
that is $\sigma_m(q,r)=\frac N2(\frac1q-\frac1r)+\frac m2$, by
\begin{align} \label{eq:reasonC2}
&2<1+b+1/p<\sigma_1(q_0,q_{03})+\sigma_0(q_0,q_{04})
\\ &\leq \frac N2(\frac1{q_0}-\frac1{q_{03}})+\frac12+\frac N2(\frac1{q_0}-\frac1{q_{04}})
=\frac N{2q_0}+\frac12<\frac{N+1}{2}. 
\end{align}
\end{Rem}

We continue the estimate of $\bg(\bu)$, 
in particular, $\sum_{i=0,2}\|\del_tE_T[<t>^b\bg(\bu)]\|_{L_{p}(\BR;L_{q_i}(\Omega))}$. 
Because the definition \eqref{def:nonlin} of the nonlinear terms implies
\begin{align} \label{eq:repgut}
\del_t\bg(\bu)
=\sum_{j,k=1}^N \del_{(j,k)}\bV^2\left(\int_0^t\nabla\bu\,ds\right)\del_ju_k\otimes \bu 
+\bV^2\left( \int_0^t \nabla\bu\,ds \right)\del_t\bu 
\end{align}
for $i=0,2$, 
the estimate \eqref{eq:EstiET} of $E_T$ and Lemma \ref{Lem:nonlinesti0} (d) implies
\begin{align} \label{eq:estibgut}
\begin{aligned}
& \|\del_tE_T[<t>^b\bg(\bu)]\|_{L_{p}(0,\infty;L_{q_i}(\Omega))}
\\ &\leq C(\|(\del_t<t>^b)\bg(\bu)\|_{L_{p}(0,T;L_{q_i}(\Omega))}
+ \|<t>^b\del_t\bg(\bu)\|_{L_{p}(0,T;L_{q_i}(\Omega))})
\\ &\leq C \bigg(\|bt<t>^{b-2}\bg(\bu)\|_{L_{p}(0,T;L_{q_i}(\Omega))}
+ \left\|<t>^b \nabla\bu\otimes\bu \right\|_{L_{p}(0,T;L_{q_i}(\Omega))}
\\ &+ \left\|<t>^b\int_0^t\nabla\bu\,ds \otimes \del_t\bu\right\|_{L_{p}(0,T;L_{q_i}(\Omega))}\bigg). 
\end{aligned}
\end{align}
The first term can be estimated from $bt<t>^{b-1}\leq b<t>^b$ and \eqref{esti:diu} 
and, also, the estimate for the third term was obtained 
in the estimate of $\bff(\bu)$ (see \eqref{eq:estifua} and \eqref{eq:estifuc}). 
The second term can be estimated as follows 
For $i=0,2$, by H\"{o}lder's inequality \eqref{eq:Holder} and Lemma \ref{Lem:nonlinesti0}, 
\begin{align} \label{esti:duu}
\begin{aligned}
&\left\|<t>^b\nabla\bu\otimes \bu \right\|_{L_{p}(0,T;L_{q_i}(\Omega))}
\\ &\leq \left\|<t>^b\nabla\bu\right\|_{L_{p}(0,T;L_{\infty}(\Omega))} 
\left\|\bu\right\|_{L_{\infty}(0,T;L_{q_{i}}(\Omega))}
\leq C[\bu]_{(0,T)}^2. 
\end{aligned}
\end{align}
To summarize, we have 
\begin{align}
\sum_{i=0,2}\|<t>^b(E_T\bg(\bu),\del_t E_T\bg(\bu))\|_{L_{p}(0,T;L_{q_i}(\Omega))}
\leq C[\bu]_{(0,T)}^2. 
\end{align}

\subti{Estimate for $g(\bu)$ and $\bh(\bu)$}

We prove 
\begin{align} \label{eq:estighu}
&\sum_{i=0,2}(\|\tdh E_T[<t>^b(g(\bu),\bh(\bu)\bn)]\|_{L_{p}(\BR;L_{q_i}(\Omega))} 
\\ &\qquad+\|E_T[<t>^b(g(\bu),\bh(\bu)\bn)]\|_{L_{p}(\BR;H^1_{q_i}(\Omega))}) 
\leq C[\bu]_{(0,T)}^2. 
\end{align}
To estimate the first term of the left-hand side, 
we introduce an extension mapping $\iota:L_{1,\loc}(\Omega)\to L_{1,\loc}(\Omega)$ satisfying 
\begin{itemize}
\item[(e1)] For any $q\in(1,\infty)$ and $f\in H^m_q(\Omega)$,$\iota f\in H^{m}_{q}(\BR^N)$ 
and $\|\iota f\|_{H^{m}_{q}(\BR^N)}\leq C\|f\|_{H^m_{q}(\Omega)}$ hold for $m=0,1$. \item[(e2)] For any $q\in(1,\infty)$ and $f\in H^1_q(\Omega)$, 
$\|(1-\Delta)^{-1/2}\iota(\nabla f)\|_{L_{q}(\BR^N)}\leq C\|f\|_{L_{q}(\Omega)}$,  
where the operator $(1-\Delta)^s$ is defined by 
$(1-\Delta)^sg=\IFT{[(1+|\xi|^2)^s\FT{[g]}]}$, holds for $s\in\BR$. 
\end{itemize}
Then, by the same fashion as in \cite[Appendix A]{Shiba15DE}, for $p,q\in(1,\infty)$, 
we have 
\begin{align} \label{eq:embedhalf}
\begin{aligned}
&H^{1}_{p}(\BR;H^{-1}_{q}(\Omega)) \cap L_{p}(\BR;H^{1}_{q}(\Omega)) 
\subset H^{1/2}_{p}(\BR;L_{q}(\Omega)), 
\\ & \|\tdh f\|_{L_{p}(\BR;L_{q}(\Omega))}
\leq C(\|\del_t[(1-\Delta)^{-1/2}\iota f]\|_{L_{p}(\BR;L_{q}(\BR^N))}
+\|f\|_{L_{p}(\BR;H^1_{q}(\Omega))}),  
\end{aligned}
\end{align}
where 
\begin{align}
H^{-1}_q(\Omega)=\text{ dual of } \widehat{H}^{1}_{q',0}(\Omega). 
\end{align}
Thus, by the embedding \eqref{eq:embedhalf} and the estimate \eqref{eq:EstiET} of $E_T$, for $i=0,2$, 
\begin{align}
&\|\tdh E_T[<t>^b(g(\bu),\bh(\bu)\bn)]\|_{L_{p}(\BR;L_{q_i}(\Omega))} 
\\ &\leq C(\|\del_t[(1-\Delta)^{-1/2}\iota E_T[<t>^b(g(\bu),\bh(\bu)\bn)]]\|_{L_{p}(\BR;L_{q_i}(\Omega))}
\\ &\qquad +\|E_T[<t>^b(g(\bu),\bh(\bu)\bn)]\|_{L_{p}(\BR;H^{1}_{q_i}(\Omega))})
\\ &\leq C(\|\del_t[(1-\Delta)^{-1/2}\iota <t>^b(g(\bu),\bh(\bu)\bn)]\|_{L_{p}(0,T;L_{q_i}(\Omega))}
\\ &\qquad +\|<t>^b(g(\bu),\bh(\bu)\bn)\|_{L_{p}(0,T;H^{1}_{q_i}(\Omega))}).
\end{align}
and so, by \eqref{eq:EstiET} again, 
\begin{align} \label{eq:estighua}
\begin{aligned}
&\sum_{i=0,2}(\|\tdh E_T[<t>^b(g(\bu),\bh(\bu)\bn)]\|_{L_{p}(\BR;L_{q_i}(\Omega))} 
\\ &\qquad+\|E_T[<t>^b(g(\bu),\bh(\bu)\bn)]\|_{L_{p}(\BR;H^1_{q_i}(\Omega))}) 
\\ &\leq C(\|\del_t[(1-\Delta)^{-1/2}\iota <t>^b(g(\bu),\bh(\bu)\bn)]\|_{L_{p}(0,T;L_{q_i}(\Omega))}
\\ &\qquad +\|<t>^b(g(\bu),\bh(\bu)\bn)\|_{L_{p}(0,T;H^{1}_{q_i}(\Omega))}). 
\end{aligned}
\end{align}

To estimate the first term, 
we apply the following lemma with $(f_1,f_2,g)=(f_1^k,f_2^k,g^k)$ for $k=1,2$ by setting 
\begin{align} \label{eq:deff1f2g}
\begin{aligned}
&(f_1^1,f_2^1,g^1)=\left(<t>^b\bV^3\left(\int_0^t\nabla\bu\,ds\right),1,\bu\right), 
\\ &(f_1^2,f_2^2,g^2)=\left(<t>^b\bV^4\left(\int_0^t\nabla\bu\,ds\right),\bn,\bu\right) 
\end{aligned}
\end{align}
because $g(\bu)=f_1^1f_2^1g^1$ and $\bh(\bu)\nml=f_1^2f_2^2g^2$ 
owing to the definition \eqref{def:nonlin} of the nonlinear terms. 

\begin{Lem} \label{Lem:minusderivative}
Let $1<p,q,r_1,r_2,r_3<\infty$ satisfy 
\begin{align} \label{eq:ri:Lem:minusderivative}
\frac{1}{q}\leq\frac{1}{r_i}\leq\frac{1}{q}+\frac{1}{N} \quad (i=1,2,3)
\end{align}
and $\iota$ be the extention map introduced above. 
Then, for $f_1,f_2,g\in L_{1,\loc}(\Omega)$ and $f=f_1f_2$, the following estimate holds.
\begin{align}
&\|\del_t[(1-\Delta)^{-1/2}\iota(f\nabla g)]\|_{L_{p}(0,T;L_{q}(\BR^N))}
\\ &\leq C(\|f\del_tg\|_{L_{p}(0,T;L_{q}(\Omega))}
+\|(\del_tf)\nabla g\|_{L_{p}(0,T;L_{r_1}(\Omega))} 
\\ &\qquad +\|(\nabla f_1)f_2\del_tg\|_{L_{p}(0,T;L_{r_2}(\Omega))}
+\|f_1(\nabla f_2)\del_tg\|_{L_{p}(0,T;L_{r_3}(\Omega))}). 
\end{align}
\end{Lem}
\begin{proof}
We follow the idea in the proof of \cite[Lemma 3.3]{Shiba15DE}. We rewrite 
\begin{align}
&\del_t[(1-\Delta)^{-1/2}\iota(f\nabla g)]
\\ &=(1-\Delta)^{-1/2}\iota(\del_tf\nabla g)
+(1-\Delta)^{-1/2}\iota(f\nabla\del_tg)
\\ &=(1-\Delta)^{-1/2}\iota(\del_tf\nabla g)
+(1-\Delta)^{-1/2}\iota [\nabla(f\del_tg)]
-(1-\Delta)^{-1/2}\iota(\nabla f\del_tg). 
\\ &=(1-\Delta)^{-1/2}\iota(\del_tf\nabla g)
+(1-\Delta)^{-1/2}\iota [\nabla(f\del_tg)]
\\ &-(1-\Delta)^{-1/2}\iota[(\nabla f_1)f_2\del_tg] 
-(1-\Delta)^{-1/2}\iota[f_1(\nabla f_2)\del_tg],. 
\end{align}
and then we obtain 
\begin{align}
&\|\del_t[(1-\Delta)^{-1/2}\iota(f\nabla g)]\|_{L_{p}(0,T;L_{q}(\BR^N))}
\\ &\leq C(\|(1-\Delta)^{-1/2}\iota(\del_tf\nabla g)\|_{L_{p}(0,T;L_{q}(\BR^N))}
\\ &\qquad +\|(1-\Delta)^{-1/2}\iota [\nabla(f\del_tg)]\|_{L_{p}(0,T;L_{q}(\BR^N))}
\\ &\qquad +\|(1-\Delta)^{-1/2}\iota[(\nabla f_1)f_2\del_tg]\|_{L_{p}(0,T;L_{q}(\BR^N))}
\\ &\qquad +\|(1-\Delta)^{-1/2}\iota[f_1(\nabla f_2)\del_tg]\|_{L_{p}(0,T;L_{q}(\BR^N))}). 
\end{align}
The second term is estimated by the property (e2) of the extension mapping $\iota$ as 
\begin{align}
\|(1-\Delta)^{-1/2}\iota [\nabla(f\del_tg)]\|_{L_{p}(0,T;L_{q}(\BR^N))}
\leq \|f\del_tg\|_{L_{p}(0,T;L_{q}(\Omega))} 
\end{align}
and, from Sobolev's inequality, the other terms are estimated as follows. 
\begin{align*}
&\|(1-\Delta)^{-1/2}\iota(\del_tf\nabla g)\|_{L_{p}(0,T;L_{q}(\BR^N))}
\\ &\leq C\|(1-\Delta)^{-1/2}\iota(\del_tf\nabla g)\|_{L_{p}(0,T;H^1_{r_1}(\BR^N))}
\leq C\|\del_tf\nabla g\|_{L_{p}(0,T;L_{r_1}(\Omega))},
\\ &\|(1-\Delta)^{1/2}\iota[(\nabla f_1)f_2\del_tg]\|_{L_{p}(0,T;L_{q}(\BR^N))}
\\ &\leq C\|(1-\Delta)^{-1/2}\iota[(\nabla f_1)f_2\del_tg]\|_{L_{p}(0,T;H^1_{r_2}(\BR^N))}
\leq C\|(\nabla f_1)f_2\del_tg\|_{L_{p}(0,T;L_{r_2}(\Omega))}, 
\\ &\|(1-\Delta)^{1/2}\iota[f_1(\nabla f_2)\del_tg]\|_{L_{p}(0,T;L_{q}(\BR^N))}
\\ &\leq C\|(1-\Delta)^{-1/2}\iota[f_1(\nabla f_2)\del_tg]\|_{L_{p}(0,T;H^1_{r_3}(\BR^N))}
\leq C\|f_1(\nabla f_2)\del_tg\|_{L_{p}(0,T;L_{r_3}(\Omega))}. 
\end{align*}
This completes the proof. 
\end{proof}

{\bf Estimate for the first term of the right-hand side in \eqref{eq:estighua}}

Define $(f_1^k,f_2^k,g^k)$ for $k=1,2$ by \eqref{eq:deff1f2g} 
so that $g(\bu)=f_1^1f_2^1g^1$ and $\bh(\bu)\nml=f_1^2f_2^2g^2$.
Set $f^k=f_1^kf_2^k$. 
Then, by Lemma \ref{Lem:minusderivative} with $(f_1,f_2,g,f)=(f_1^k,f_2^k,g^k,f^k)$, 
the first term of the right-hand side in \eqref{eq:estighua} is estimated as 
\begin{align} \label{eq:estighu_applLem}
\begin{aligned}
&\sum_{i=0,2}\|\del_t[(1-\Delta)^{-1/2}\iota <t>^b(g(\bu),\bh(\bu)\bn)]
\|_{L_{p}(0,T;L_{q_i}(\Omega))}
\\ &\leq C\sum_{k=1,2,\,i=0,2}(\|f^k\del_tg^k\|_{L_{p}(0,T;L_{q_i}(\Omega))}
+\|(\del_tf^k)\nabla g^k\|_{L_{p}(0,T;L_{q_i}(\Omega))} 
\\ &\qquad +\|(\nabla f_1^k)f_2^k\del_tg\|_{L_{p}(0,T;L_{q_0}(\Omega))}
+\|(\nabla f_1^k)f_2^k\del_tg\|_{L_{p}(0,T;L_{q_2/2}(\Omega))}
\\ &\qquad +\|f_1^k(\nabla f_2^k)\del_tg^k\|_{L_{p}(0,T;L_{q_i}(\Omega))}). 
\end{aligned}
\end{align}
Noting that $f^k=f_1^kf_2^k$ and that $f_2^k$ is independent of $t$, we have 
\begin{align} \label{eq:repgutf1k}
\del_tf_1^k
&=f_2^k\del_t\bV^{2+k} \left(\int_0^t\nabla\bu\,ds\right)
\\ &=f_2^k\sum_{i,j=1}^N \del_{(i,j)}\bV^{2+k}\left(\int_0^t\nabla\bu\,ds\right)\del_iu_j\otimes \bu 
+\bV^{2+k}\left( \int_0^t \nabla\bu\,ds \right)\del_t\bu. 
\end{align}
Thus, for $q\in[1,\infty]$ and $k=1,2$, 
H\"{o}lder's inequality \eqref{eq:Holder} and Lemma \ref{Lem:nonlinesti0} (d) yield 
\begin{align}
&\|f^k\del_tg^k\|_{L_{p}(0,T;L_{q}(\Omega))}
\leq C\left\|\int_0^t\nabla\bu\,ds\otimes<t>^b\del_t\bu\right\|_{L_{p}(0,T;L_{q}(\Omega))}, 
\\ &\|(\del_tf^k)\nabla g^k\|_{L_{p}(0,T;L_{q}(\Omega))}
\\ \notag &\leq C(\left\|\int_0^t\nabla\bu\,ds\otimes<t>^b\nabla\bu\right\|_{L_{p}(0,T;L_{q}(\Omega))}
+\|\nabla\bu\otimes<t>^b\nabla\bu\|_{L_{p}(0,T;L_{q}(\Omega))}), 
\\ &\|(\nabla f_1^k)f_2^k\del_tg^k\|_{L_{p}(0,T;L_{q}(\Omega))}
\leq C\left\|\int_0^t\nabla^2\bu\,ds\otimes<t>^b\del_t\bu\right\|_{L_{p}(0,T;L_{q}(\Omega))}, 
\\ &\|f_1^k(\nabla f_2^k)\del_tg^k\|_{L_{p}(0,T;L_{q}(\Omega))}
\leq C\left\|\int_0^t\nabla\bu\,ds\otimes<t>^b\del_t\bu\right\|_{L_{p}(0,T;L_{q}(\Omega))} 
\end{align} 
and so, by combining these estimates with \eqref{eq:estighu_applLem}, we have 
\begin{align} \label{eq:estighub}
\begin{aligned}
&\|\del_t[(1-\Delta)^{-1/2}\iota <t>^b(g(\bu),\bh(\bu)\bn)]
\|_{L_{p}(0,T;L_{q_0}(\Omega))}
\\ &\leq C\sum_{i=0,2} \bigg(
\left\|\int_0^t\nabla\bu\,ds\otimes<t>^b\del_t\bu\right\|_{L_{p}(0,T;L_{q_i}(\Omega))}
\\ &\qquad+\left\|<t>^{b}\int_0^t\nabla\bu\,ds\otimes\nabla\bu\right\|_{L_{p}(0,T;L_{q_i}(\Omega))}
+\left\|\nabla\bu\otimes<t>^b\nabla\bu\right\|_{L_{p}(0,T;L_{q_i}(\Omega))}
\\ &\qquad+\left\|\int_0^t\nabla^2\bu\,ds\otimes<t>^b\del_t\bu\right\|_{L_{p}(0,T;L_{q_0}(\Omega))}
+\left\|\int_0^t\nabla^2\bu\,ds\otimes<t>^b\del_t\bu\right\|_{L_{p}(0,T;L_{q_2/2}(\Omega))}\bigg). 
\end{aligned}
\end{align}

The first term has been estimated by $C[\bu]_{(0,T)}^2$ in \eqref{eq:estifua} and \eqref{eq:estifuc}. 
Here, we estimate the second term. 
When $i=0$, by H\"{o}lder's inequality \eqref{eq:Holder} and Lemma \ref{Lem:nonlinesti0}, 
\begin{align} \label{esti:didu}
&\left\|<t>^b\int_0^t\nabla\bu\,ds \otimes \nabla\bu \right\|_{L_{p}(0,T;L_{q_0}(\Omega))}
\\ &\leq \left\|\int_0^t\nabla\bu\,ds\right\|_{L_{\infty}(0,T;L_{q_{03}}(\Omega))}
\left\|<t>^b\nabla\bu\right\|_{L_{p}(0,T;L_{q_{04}}(\Omega))} 
\leq C[\bu]_{(0,T)}^2. 
\end{align}
The estimate for $i=2$ can be obtained 
by replacing $q_{03}$ and $q_{04}$ with $q_2$ and $\infty$, respectively.  
The third term of the right-hand side in \eqref{eq:estighub} can be estimated 
by H\"{o}lder's inequality \eqref{eq:Holder} and Lemma \ref{Lem:nonlinesti0} as follows: for $i=0,2$, 
\begin{align} \label{esti:dudu}
\begin{aligned}
&\left\|<t>^b\nabla\bu \otimes \nabla\bu \right\|_{L_{p}(0,T;L_{q_i}(\Omega))}
\\ &\leq C\|<t>^b\nabla\bu\|_{L_{p}(0,T;L_{\infty}(\Omega))}
\|\nabla\bu\|_{L_{\infty}(0,T;L_{q_i}(\Omega))}
\leq C[\bu]_{(0,T)}^2. 
\end{aligned}
\end{align}
We next estimate the fourth term of the right-hand side in \eqref{eq:estighub}. 
By H\"{o}lder's inequality \eqref{eq:Holder} and Lemma \ref{Lem:nonlinesti0},  
\begin{align} \label{esti:ddiddu0}
\begin{aligned}
&\left\|<t>^b\int_0^t\nabla^2\bu\,ds \otimes \del_t\bu \right\|_{L_{p}(0,T;L_{q_0}(\Omega))}
\\ &\leq \left\|\int_0^t\nabla^2\bu\,ds\right\|_{L_{\infty}(0,T;L_{q_{03}}(\Omega))}
\left\|<t>^b\del_t\bu\right\|_{L_{p}(0,T;L_{q_{04}}(\Omega))} 
\leq C[\bu]_{(0,T)}^2. 
\end{aligned}
\end{align}
Finally, the fifth term of the right-hand side in \eqref{eq:estighub} can be estimated 
by H\"{o}lder's inequality \eqref{eq:Holder} and Lemma \ref{Lem:nonlinesti0} as 
\begin{align} \label{esti:ddiddu2}
\begin{aligned}
&\left\|<t>^b\int_0^t\nabla^2\bu\,ds \otimes \del_t\bu \right\|_{L_{p}(0,T;L_{q_2/2}(\Omega))}
\\ &\leq \left\|\int_0^t\nabla^2\bu\,ds\right\|_{L_{\infty}(0,T;L_{q_{2}}(\Omega))}
\left\|<t>^b\del_t\bu\right\|_{L_{p}(0,T;L_{q_{2}}(\Omega))} 
\leq C[\bu]_{(0,T)}^2. 
\end{aligned}
\end{align}

{\bf Estimate of the second term in \eqref{eq:estighua}}

It remains to show the estimate for the second term 
$\|<t>^b(g(\bu),\bh(\bu)\bn)\|_{L_{p}(0,T;H^{1}_{q_i}(\Omega))}$
of the right-hand side in \eqref{eq:estighua}. 
By the definition \eqref{def:nonlin} of the nonlinear terms 
and Lemma \eqref{Lem:nonlinesti0} (d), for $i=0,2$, 
\begin{align} \label{eq:estighu2ndt}
\begin{aligned}
&\|<t>^b(g(\bu),\bh(\bu)\bn)\|_{L_{p}(0,T;H^{1}_{q_i}(\Omega))}
\\ &\leq C \bigg(\left\|<t>^b\int_0^t\nabla\bu\,ds\otimes\nabla\bu\right\|_{L_{p}(0,T;L_{q_i}(\Omega))}
\\ &\qquad+ \left\|<t>^b\int_0^t\nabla^2\bu\,ds\otimes\nabla\bu\right\|_{L_{p}(0,T;L_{q_i}(\Omega))}
\\ &\qquad+ \left\|<t>^b\int_0^t\nabla\bu\,ds\otimes\nabla^2\bu\right\|_{L_{p}(0,T;L_{q_i}(\Omega))}\bigg),  
\end{aligned}
\end{align}
and so, 
since we have estimated these terms in \eqref{esti:didu}, \eqref{esti:ddidu} and \eqref{eq:estifua}, 
we obtain
\begin{align}
\|<t>^b(g(\bu),\bh(\bu)\bn)\|_{L_{p}(0,T;H^{1}_{q_i}(\Omega))} \leq C[\bu]_{(0,T)}^2. 
\end{align}

To summarize, we conclude \eqref{esti:nonlin}.


\section{Proof of Theorem \ref{Thm:GWPha}} \label{sec:half}

In this section, we prove the global well-posedness 
and the decay property in $\BR^N_+$ with $N\geq3$. 
To obtain the global well-posedness, for $T>0$, 
assuming the unique existence of a solution $\bu$ of \eqref{LNS} with $\Omega=\hsp$ on $(0,T)$, 
we show an analogue of the estimate \eqref{esti:lin0} for the Stokes problem 
and an analogue of the estimate \eqref{esti:nonlin} for the nonlinear term. 

The difficulty in obtaining the result for $N=3$ arises from the estimate of 
$\|<t>^b\bg(\bu)\|_{L_{p}(0,T;L_{q_0}(\Omega))}$ 
beacause $\bg(\bu)$ has the slowest decay of the nonliner terms. 
However, this term is just an additional term appearing 
when we apply the maximal regularity to the time-shifted Stokes problem \eqref{SS}. 
This observation enables us to overcome this difficulty 
by reducing the Stokes problem to the problem with $(G,\bG)=(0,\fb0)$
before applying it; 
we subtract a function $K_0G$ satisfying $\divv K_0G=G=\divv\bG$ from the solution. 
The function $K_0G$ is constructed by solving the Poisson equation 
and, thanks to $\Omega=\BR^N_+$, 
$K_0G$ has homogeneous estimates with respect to the derivative order, 
see \eqref{esti:diveq:Lem} below. 
Owing to this, in Subsection \ref{subsec:linesti:half}, 
we obtain the estimate 
\begin{align} \label{esti:linn:half}
[\bu]_{(0,T)} \leq 
C(\cc{N}_{\hsp}(\bff(\bu),\bg(\bu),g(\bu),\bh(\bu)\bn,\bv_0)+[K_0g(\bu)]_{(0,T)}) 
\end{align}
in $\hsp$, which corresponds to \eqref{esti:linn} 
but the first term of the right-hand side does not include the crucial norm 
$\|<t>^b\bg(\bu)\|_{L_{p}(0,T;L_{q_0}(\Omega))}$. 
Here,  
\begin{align} \label{def:NormofData:half}
\begin{aligned}
&\cc{N}_{\hsp}(\bff(\bu),\bg(\bu),g(\bu),\bh(\bu),\bv_0) 
\\ &\begin{aligned}
= \sum_{i=0,2}\big(&\|<t>^b\bff(\bu)\|_{L_{p}(0,T;L_{q_i}(\Omega))}
+\|\del_tE_T[<t>^b\bg(\bu)]\|_{L_{p}(\BR;L_{q_i}(\Omega))}
\\ & +\|\tdh E_T[<t>^b(g(\bu),\bh(\bu)\bn)]\|_{L_{p}(\BR;L_{q_i}(\Omega))}
\\ & +\|E_T[<t>^b(g(\bu),\bh(\bu)\bn)]\|_{L_{p}(\BR;H^1_{q_i}(\Omega))}
+\|\bv_0\|_{B^{2(1-1/p)}_{q_i,p}(\Omega)}\big). 
\end{aligned}
\end{aligned}
\end{align}
In Subsection \ref{subsec:nonlinesti:half}, 
we show that the second term $[K_0g(\bu)]_{(0,T)}$ is harmless 
and obtain the analogue of \eqref{esti:nonlin}, 
\begin{align} \label{esti:nonlin:half}
\cc{N}_{\hsp}(\bff(\bu),\bg(\bu),g(\bu),\bh(\bu),\bv_0)+[K_0g(\bu)]_{(0,T)} 
\leq C(\cc{I}+[\bu]_{(0,T)}^2), 
\end{align} 
where the norm $\cc{I}$ of $\bv_0$ is defined by \eqref{def:ininorm}.  
Then, we have 
\begin{align} \label{esti:half}
[\bu]_{(0,T)}\leq C\eps 
\end{align}
and, in the same way as in \eqref{eq:decaypf}, we obtain the decay property \eqref{eq:decay}.

\subsection{The $L_q$-$L_r$ estimates} \label{subsec:ass:half}
In this subsection, to employ the same argument as in Section \ref{sec:gen}, 
we prove the $L_q$-$L_r$ estimates for the decay rate 
\begin{align} \label{eq:decayLqLr:half}
\sigma_m(q,r)=\frac N2 \left(\frac1q-\frac1r\right)+\frac m2
\end{align} 

\begin{Prop} \label{Prop:LqLresti:half}
Define $\sigma_m(q,r)$ by \eqref{eq:decayLqLr:half}. 
Then, for $(q,r)$ satisfying $1<q\leq r\leq\infty$ and $q\ne\infty$, 
there exists $C=C(q,r)>0$ such that 
\begin{align} \label{esti:LqLr:half}
\begin{aligned}
\| (\del_te^{-tA_q}\bff,\nabla^2e^{-tA_q}\bff) \|_{L_{r}(\hsp)} 
&\leq Ct^{-\sigma_2(q,r)}\|\bff\|_{L_{q}(\hsp)} \quad (r\ne\infty) \\ 
\| \nabla^m e^{-tA_q}\bff \|_{L_{r}(\hsp)} 
&\leq Ct^{-\sigma_m(q,r)}\|\bff\|_{L_{q}(\hsp)} \quad (m=0,1) 
\end{aligned}
\end{align}
for $t\geq1$ and $\bff\in J_{q}(\hsp)$. 
\end{Prop}
\begin{proof}
\underline{\bf The case $r=q$).} 
We first consider the case $r=q$. 
Let $\phi\in(0,\pi/2)$, $1<q<\infty$ and $\bff\in J_q(\hsp)$. 
By the resolvent estimates for the resolvent Stokes equation in $\hsp$ 
obtained in \cite[Theorem 4.1]{ShibaShimi03DIE}, 
$\bu_\lambda=(\lambda+A_q)^{-1}\bff$ satisfies
\begin{align} \label{eq:RE:half}
|\lambda|\|\bu_\lambda\|_{L_{q}(\hsp)}
+|\lambda|^{1/2}\|\nabla\bu_\lambda\|_{L_{q}(\hsp)}
+\|\nabla^2\bu_\lambda\|_{L_{q}(\hsp)} \leq C\|\bff\|_{L_{q}(\hsp)} 
\end{align}
for $\lambda\in\bb{C}\setminus\{0\}$ with $|\arg\lambda|<\pi-\phi/2$. 
By this and the properties of analytic semigroup, we obtain 
\begin{align} \label{esti:LqLq1:half}
\begin{aligned}
\| \del_te^{-tA_q}\bff \|_{L_{r}(\hsp)} 
&\leq Ct^{-\sigma_2(q,q)}\|\bff\|_{L_{q}(\hsp)}, \\ 
\| e^{-tA_q}\bff \|_{L_{r}(\hsp)} 
&\leq Ct^{-\sigma_0(q,q)}\|\bff\|_{L_{q}(\hsp)} 
\end{aligned}
\end{align}
Also, the resolvent estimate \eqref{eq:RE:half} 
and the change of variable $\lambda t=\lambda'$ in the formula 
\begin{align} \label{eq:sgform}
e^{-tA_q}\bff = \int_\Gamma e^{\lambda t}(\lambda+A_q)^{-1}\bff \,d\lambda,  
\end{align}
where $\Gamma$ is suitable contour 
from $\infty e^{-(\pi-\phi)i}$ to $\infty e^{(\pi-\phi)i}$, yield 
\begin{align} \label{esti:LqLq2:half}
\| \nabla^m e^{-tA_q}\bff \|_{L_{r}(\hsp)} 
&\leq Ct^{-\sigma_m(q,q)}\|\bff\|_{L_{q}(\hsp)} \quad (m=1,2). 
\end{align}
This and \eqref{esti:LqLq1:half} imply \eqref{esti:LqLr:half}. 

\underline{\bf The case $0<1/r-1/q\leq 1/N$)}. 
We now show the result for the case $0<1/r-1/q\leq 1/N$ 
from the Gagliardo-Nirenberg interpolation inequality. 
If we define the even extension operator $E^e:L_{1,\loc}(\wsp)\to L_{1,\loc}(\hsp)$ as  
\begin{align}
E^ef(x',x_N)=\left\{ \begin{array}{ll}
f(x',x_N) & x_N>0, \\
f(x',-x_N) & x_N<0 
\end{array} \right.
\end{align}
and set 
\begin{align}
\alpha=N\left(\frac{1}{q}-\frac{1}{r}\right), 
\end{align}
from the Gagliardo-Nirenberg interpolation inequality 
and the result for the case $r=q$, 
\begin{align*}
\|\nabla^me^{-tA_q}\bff\|_{L_{r}(\hsp)} 
&\leq \|E^e\nabla^me^{-tA_q}\bff\|_{L_{r}(\BR^N)} 
\\ &\leq C\|\nabla E^e\nabla^me^{-tA_q}\bff\|_{L_{q}(\BR^N)}^{\alpha}
\|E^e\nabla^me^{-tA_q}\bff\|_{L_{q}(\BR^N)}^{1-\alpha}
\\ &\leq C\|\nabla^{m+1}e^{-tA_q}\bff\|_{L_{q}(\hsp)}^{\alpha}
\|\nabla^me^{-tA_q}\bff\|_{L_{q}(\hsp)}^{1-\alpha}
\\ &\leq Ct^{-\frac{m+1}{2}\alpha-\frac{m}{2}(1-\alpha)}\|\bff\|_{L_{q}(\hsp)}
\\ &=Ct^{-\frac{N}{2}\left(\frac{1}{q}-\frac{1}{r}\right)-\frac{m}{2}}\|\bff\|_{L_{q}(\hsp)}
\end{align*}
for $m=0,1$. By combining this with the result for $r=q\ne\infty$, we also obtain 
\begin{align}
&\|\del_te^{-tA_q}\bff\|_{L_{r}(\hsp)} + \|\nabla^2e^{-tA_q}\bff\|_{L_{r}(\hsp)}
\\ &\leq Ct^{-1}\|e^{-(t/2)A_q}\bff\|_{L_{r}(\hsp)}
\leq Ct^{-\frac{N}{2}\left(\frac{1}{q}-\frac{1}{r}\right)-1}\|\bff\|_{L_{q}(\hsp)}. 
\end{align}
\underline{\bf The case $1/r-1/q>1/N$)}. 
The estimate for the case $1/r-1/q>1/N$ can be obtained  
by repeating use of the result for the case $1/r-1/q\leq 1/N$. 
\end{proof}

\subsection{Estimate for the Stokes problem in the half space} \label{subsec:linesti:half}

In this subsection, we show a theorem analogous to Theorem \ref{Thm:MR} for $\Omega=\hsp$ 
by reducing the Stokes problem to the problem with $(\bG,G)=(\fb0,0)$. 
To state the theorem and to execute the reduction, 
we introduce a solution operator $K_0$ to the divergence equation, 
which is proved for example in \cite[Lemma 4.1 (1)]{ShibaShimi12MSJ}
by solving the Poisson equation. 
\begin{Lem}[{e.g. \cite{ShibaShimi12MSJ}}] \label{Lem:diveq}
Let $1<q<\infty$. 
There exists an operator 
$K_0:H^1_q(\BR^N_+)\cap\widehat{H}^{-1}_q(\BR^N_+)\to H^{2}_{q}(\hsp)^N$ 
such that, for $g\in H^1_q(\BR^N_+)\cap\widehat{H}^{-1}_q(\BR^N_+)$, 
$K_0g$ satisfies the divergence equation $\divv v=g$ and the estimate 
\begin{align} \label{esti:diveq:Lem}
\begin{gathered}
\|K_0g\|_{L_{q}(\hsp)}\leq C\|g\|_{\widehat{H}^{-1}_{q}(\hsp)}, \quad 
\|\nabla K_0g\|_{L_{q}(\hsp)}\leq C\|g\|_{L_{q}(\hsp)}, \\ 
\|\nabla^2K_0g\|_{L_{q}(\hsp)}\leq C\|\nabla g\|_{L_{q}(\hsp)}. 
\end{gathered}
\end{align}
\end{Lem}

The following theorem is the main theorem of this subsection. 
Note that \eqref{esti:linn:half} is obtained immediately by \eqref{esti:MR:half} below 
if we assume the unique existence of the solution $\bu$ to \eqref{LNS} on $(0,T)$.  
\begin{Thm} \label{Thm:MR:half}
Let $1<p,q<\infty$ and $T\in(0,\infty]$. 
Define $\bmmpq$ by \eqref{def:bmpq} 
for the decay rate $\sigma_m(q,r)$ and $\delta$ 
given by \eqref{eq:decayLqLr:half} and \eqref{def:deltahalf}, respectively. 
For any $\bv_0\in\cc{D}_{q,p}(\hsp)$ 
and right members $(\bF,\bG,G,\bH)$ defined on $(0,T)$ satisfying 
\begin{align} \label{eq:DataClassMR:half}
\begin{aligned}
&<t>^b\bF \in L_{p}(0,T;L_{q}(\hsp)^N), \quad  
E_T[<t>^b\bG] \in H^1_{p}(\BR;L_{q}(\hsp)^N), 
\\ &E_T[<t>^bG] \in H^{1/2}_{p}(\BR;L_{q}(\hsp))\cap L_{p}(\BR;H^{1}_{q}(\hsp)), 
\\ &E_T[<t>^b\bH] \in H^{1/2}_{p}(\BR;L_{q}(\hsp)^N)\cap L_{p}(\BR;H^{1}_{q}(\hsp)^N), 
\end{aligned}
\end{align}
with the compatibility condition
\begin{align} \label{eq:DataCCMR:half}
(G(t),\varphi)_{\hsp} = -(\bG(t),\nabla\varphi)_{\hsp} 
\text{ for any } \varphi\in\widehat{H}^{1}_{q',0}(\hsp), 
\end{align}
the Stokes problem \eqref{S} admits unique solutions 
\begin{gather}
\bU \in H^{1}_{p}(0,T;L_{q}(\hsp)^N) \cap L_{p}(0,T;H^{2}_{q}(\hsp)^N), \ 
\Pres \in L_{p}(0,T;\sppresh). 
\end{gather} 
Moreover, the solutions possess the estimate
\begin{align} \label{esti:MR:half}
[\bu]_{(0,T)} \leq C_b\cc{N}_{\hsp}(\bF,\bG,G,\bH,\bv_0) 
\end{align}
for $b\geq0$, where the constant $C_b>0$ is independent of $T$. 
\end{Thm}

\begin{proof}
It suffices to construct a solution of \eqref{S} 
possessing the estimate 
\begin{align} \label{esti:lin0:half}
\begin{aligned}
&\|<t>^{\bmmpq[2]}\del_t\bU\|_{L_{\tp_2}(0,T;L_{\tq_2}(\hsp))} 
\\ &\leq C_b\cc{N}_{\hsp}(\bF,\bG,G,\bH,\bv_0)+\|<t>^{\bmmpq[2]}\del_tK_0G\|_{L_{\tp_2}(0,T;L_{\tq_2}(\hsp))}, 
\\ &\|<t>^{\bmmpq}\nabla^m\bU\|_{L_{\tp_m}(0,T;L_{\tq_m}(\hsp))} 
\\ &\leq C_b\cc{N}_{\hsp}(\bF,\bG,G,\bH,\bv_0)+\|<t>^{\bmmpq}\nabla^mK_0G\|_{L_{\tp_m}(0,T;L_{\tq_m}(\hsp))}, 
\end{aligned}
\end{align}
for any $m=0,1,2$ and $(\tp_m,\tq_m)$ satisfying \eqref{eq:exp_restriction}. 
For almost everywhere $t\in(0,T)$, 
because the compatibility condition \eqref{eq:DataCCMR:half} yields
\begin{align} \label{eq:GinH-1q}
|(G(t),\varphi)_{\hsp}|=|(\bG(t),\nabla\varphi)_{\hsp}|
\leq\|\bG(t)\|_{L_{q}(\hsp)}\|\nabla\varphi\|_{L_{q'}(\hsp)} 
\text{ for any } \varphi\in \widehat{H}^{1}_{q',0}(\hsp), 
\end{align}
we have $G(t)\in\widehat{H}^{-1}_{q'}(\hsp)$ 
and $\|G(t)\|_{\widehat{H}^{-1}_{q}(\hsp)}\leq\|\bG(t)\|_{L_{q}(\hsp)}$ 
and so, by Lemma \ref{Lem:diveq}, 
we have $\divv K_0G=G$ and 
\begin{gather} \label{esti:K0G1}
\|K_0G(t)\|_{L_{q}(\hsp)}\leq C\|G(t)\|_{\widehat{H}^{-1}_{q}(\hsp)}
\leq C\|\bG(t)\|_{L_{q}(\hsp)}, 
\\ \label{esti:K0G2}
\begin{gathered}
\|\nabla K_0G(t)\|_{L_{q}(\hsp)}\leq C\|G(t)\|_{L_{q}(\hsp)}, \\ 
\|\nabla^2K_0G(t)\|_{L_{q}(\hsp)}\leq C\|\nabla G(t)\|_{L_{q}(\hsp)}. 
\end{gathered}
\end{gather}
For the solutions $\bU$ and $\Pres$ of the Stokes problem \eqref{S}, 
if we set $\bU=K_0G+\bU_r$, $\bU_r$ and $\Pres$ obey the system
\begin{align} \label{S:G=0}
\left\{
\begin{array}{ll}
\del_t\bU_r - \Divv\bS(\bU_r,\Pres) = \bF_r, \quad
\divv \bU_r=0 & \text{in } \hsp\times(0,T), \\
\bS(\bU_r,\Pres)\nml = \bH_r & \text{on } \del\hsp\times(0,T), \\
\bU_r|_{t=0}=\bv_0 & \text{in } \hsp, 
\end{array}
\right. 
\end{align}
where the right members $\bF_r$ and $\bH_r$ are defined by 
\begin{align}
\bF_r=\bF-\del_tK_0G+\Divv(\mu\bD(K_0G)), 
\quad \bH_r=\bH-\mu\bD(K_0G)\bn. 
\end{align}
Note the right member $\bv_0$ of the initial condition does not change since 
$K_0[G]|_{t=0}=K_0[G|_{t=0}]=\fb0$. 
Since \eqref{eq:DataCCMR:half} is valid for $(G,\bG)=(\del_tG,\del_t\bG)$, 
we similarly have 
\begin{align}
\|\del_tK_0G(t)\|_{L_{q}(\hsp)}=\|K_0\del_tG(t)\|_{L_{q}(\hsp)}
\leq C\|\del_t\bG(t)\|_{L_{q}(\hsp)}
\end{align}
and then, this and \eqref{esti:K0G2} imply 
\begin{align}
& \|<t>^b\bF_r\|_{L_{p}(0,T;L_{q}(\hsp))}
\leq C\|<t>^b(\bF,\del_t\bG,\nabla G)\|_{L_{p}(0,T;L_{q}(\hsp))}, 
\\ & \|\tdh E_T[<t>^b\bH_r]\|_{L_{p}(\BR;L_{q}(\hsp))}
+\|E_T[<t>^b\bH_r]\|_{L_{p}(\BR;H^1_{q}(\hsp))}
\\ & \leq \|\tdh E_T[<t>^b(\bH,G)]\|_{L_{p}(\BR;L_{q}(\hsp))}
+\|E_T[<t>^b(\bH,G)]\|_{L_{p}(\BR;H^1_{q}(\hsp))}. 
\end{align}

In the half-space, 
Assumptions \ref{Ass:unifdom} on the $W^2_\infty$ domain is satisfied
and, on Assumption \ref{Ass:wDp}, 
the unique solvability of the weak Dirichlet problem \eqref{wDp} is well-known. 
Also, by Proposition \ref{Prop:LqLresti:half}, 
the $L_q$-$L_r$ estimates holds for the decay rate $\sigma_m(q,r)$ 
and $\sigma_m(q,r) $ satisfies the condition \textup{(C1)} in Theorem \ref{Thm:GWPgen}. 
Thus, we can apply Theorem \ref{Thm:MR} to the system \eqref{S:G=0} 
and show that \eqref{S:G=0} admits unique solutions 
\begin{align}
\bU_r \in H^{1}_{p}(0,T;L_{q}(\hsp)^N) \cap L_{p}(0,T;H^{2}_{q}(\hsp)^N), \ 
\Pres \in L_{p}(0,T;\sppresh), 
\end{align}
which possesses the estimate 
\begin{align}
&\|<t>^{\bmmpq[2]}\del_t\bU\|_{L_{\tp_2}(0,T;L_{\tq_2}(\hsp))} 
\\ &\leq C\cc{N}_{\hsp}(\bF_r,\fb0,0,\bH_r,\bv_0) 
\leq C\cc{N}_{\hsp}(\bF,\bG,G,\bH,\bv_0) 
\end{align}
and, similarly, 
$\|<t>^{\bmmpq}\nabla^m\bU\|_{L_{\tp_m}(0,T;L_{\tq_m}(\hsp))} \leq \cc{N}_{\hsp}(\bF,\bG,G,\bH,\bv_0)$. 
Combining this and $\bU=K_0G+\bU_r$ concludes 
the solvability of \eqref{S} and the estimates \eqref{esti:lin0:half}, 
which completes the proof. 
\end{proof}

\subsection{Estimate for the nonlinear terms in the half space} \label{subsec:nonlinesti:half}

In this subsection, we prove \eqref{esti:nonlin:half}. 
Let $2<p<\infty$, $1<q_0<N<q_2<\infty$, $b>1/p'$. 
Define $\bmmpq$ by \eqref{def:bmpq} 
for the decay rate $\sigma_m(q,r)$ and $\delta>0$ 
given by \eqref{eq:decayLqLr:half} and \eqref{def:deltahalf}, respectively. 
We assume that the transformed problem \eqref{LNS} admits a unique solution $\bu$ on $(0,T)$
and $[\bu]_{(0,T)}$ is sufficiently small. 
It suffices to show 
\begin{align} \label{esti:nonlin1:half}
\cc{N}_{\hsp}(\bff(\bu),\bg(\bu),g(\bu),\bh(\bu),\bv_0) \leq C(\cc{I}+[\bu]_{(0,T)}^2)  
\end{align}
and 
\begin{align} \label{esti:nonlin2:half}
\begin{aligned}
&\|<t>^{\bmmpq[2]}\del_tK_0g(\bu)\|_{L_{\tp_2}(0,T;L_{\tq_2}(\hsp))} \leq C(\cc{I}+[\bu]_{(0,T)}^2), 
\\ &\|<t>^{\bmmpq}\nabla^mK_0g(\bu)\|_{L_{\tp_m}(0,T;L_{\tq_m}(\hsp))} \leq C(\cc{I}+[\bu]_{(0,T)}^2) 
\end{aligned}
\end{align}
for $m=0,1,2$ and $(\tp_m,\tq_m)$ satisfying \eqref{eq:exp_restriction}. 
In fact, \eqref{esti:nonlin2:half} and the definition \eqref{def:Normofu} 
imply $[K_0g]_{(0,T)}\leq C(\cc{I}+[\bu]_{(0,T)}^2)$
and, by this and \eqref{esti:nonlin1:half}, we conclude \eqref{esti:nonlin:half}.

\subti{Proof of \eqref{esti:nonlin1:half}}

On account of \eqref{eq:reasonC2} in Remark \ref{Rem:C2}, if $N=3$, 
we cannot take $q_{03}$ and $q_{04}$ with $1/q_{03}+1/q_{04}=1/q_{0}$ satisfying \eqref{eq:ass:Rem}. 
Instead, we define them so that $\sigma_1(q_0,q_{03})>1$ and $\sigma_1(q_0,q_{04})>b+1/p$ by 
\begin{align} \label{def:q03q04half}
\quad 
\sigma_1(q_0,q_{03})=1+\delta_0, \ 
\sigma_1(q_0,q_{04})=b+\frac1p+\delta_0 \quad 
\text{with } \delta_0=\frac12\left(\frac{N}{2q_0}-(b+\frac1p)\right)
\quad 
\end{align}
and then, prove \eqref{esti:nonlin1:half} in the similar way as in \eqref{esti:nonlin}. 
We first state that, by the same proof, 
we have the estimates in Lemma \ref{Lem:nonlinesti0} 
except for the following estimate on the 0-th derivative of $\bu$: 
\begin{align}
\|<t>^b\bu\|_{L_{\tp_0}(0,T;L_{\tq_0}(\Omega))}\leq C[\bu]_{(0,T)} 
\quad \text{if } (\tp_0,\tq_0) \text{ satisfies \eqref{eq:exp_restriction} and } \tq_0\geq q_{04} 
\end{align}
and show that this estimate is valid if $<t>^b$ is replaced by $<t>^{b-1/2}$. 
\begin{Lem} \label{Lem:nonlinesti0:half} 
Let $q_{03}$ and $q_{04}$ be the exponents given by \eqref{def:q03q04half}. 
\begin{enumerate}
\renewcommand{\labelenumi}{(\alph{enumi})}
\setlength{\parskip}{0mm} 
\setlength{\itemsep}{0mm} 
\item There holds 
\begin{align}
\|\nabla^m\bu\|_{L_{\infty}(0,T;L_{\tq_m}(\hsp))} \leq C[\bu]_{(0,T)} 
\end{align}
if $m=0,1$ and $(\infty,\tq_m)$ satisfy \eqref{eq:exp_restriction}. 
\item There holds 
\begin{align}
&\|<t>^b\del_t\bu\|_{L_{\tp_2}(0,T;L_{\tq_2}(\hsp))} \leq C[\bu]_{(0,T)}, 
\\ &\|<t>^b\nabla^m\bu\|_{L_{\tp_m}(0,T;L_{\tq_m}(\hsp))} \leq C[\bu]_{(0,T)} \quad (m=1,2),  
\\ &\|<t>^{b-1/2}\bu\|_{L_{\tp_0}(0,T;L_{\tq_0}(\hsp))} \leq C[\bu]_{(0,T)},  
\end{align}
for $m=0,1,2$ and $(\tp_m,\tq_m)$ 
satisfying \eqref{eq:exp_restriction} and $\tq_m\geq q_{04}$. 
\item There holds 
\begin{align}
\left\| \int_0^t \nabla^m \bu(s)\,ds \right\|_{L_{\infty}(0,T;L_{\tq_m}(\hsp))} 
\leq C[\bu]_{(0,T)} 
\end{align}
if $m=1,2$ and $(p,\tq_m)$ satisfy \eqref{eq:exp_restriction} and $\tq_m\geq\min\{q_{04},q_{03}\}$. 
\item There holds $\|\bW(\int_0^t\nabla\bu\,ds)\|_{L_{\infty}(0,T;L_{\infty}(\hsp))}\leq C$ 
for any polynomial $\bW$. 
\end{enumerate}
\end{Lem}
\begin{proof}
We only need to prove the last estimate in (b). We obtain 
\begin{align}
\bmpqz{0}{p}{\tq_0} & =\min\{\sigma_0(q_0,q_{04})-\frac{1}{p}-\delta,b\}
\\ &=\min\{b+\frac1p+\delta_2-\frac12-\frac{1}{p}-\delta,b\} \geq b-\frac12 \quad (\tp_0=1/p) 
\\ \bmpqz{0}{\infty}{\tq_0} & =\min\{\sigma_0(q_0,q_{04}),b\}
\\ &=\min\{b+\frac1p+\delta_2-\frac12,b\} \geq b-\frac12 \quad (\tp_0=\infty) 
\end{align}
by \eqref{def:q03q04half}. Thus, 
\begin{align}
\|<t>^{b-1/2}\bu\|_{L_{\tp_0}(0,T;L_{\tq_0}(\hsp))} 
\leq \|<t>^{\bmmpq[0]}\bu\|_{L_{\tp_0}(0,T;L_{\tq_0}(\hsp))}
\leq C[\bu]_{(0,T)}
\end{align}
by the definition \eqref{def:bmpq} of $\bmmpq$, which implies the desired estimate. 
\end{proof}

The desired estimate \eqref{esti:nonlin1:half} is obtained as follows. 
By the estimates \eqref{eq:estifu} for $\bff(\bu)$, 
\eqref{eq:estibgut} for $\del_tE_T[<t>^n\bg(\bu)]$, 
\eqref{eq:estighua}, \eqref{eq:estighub} and \eqref{eq:estighu2ndt} for $(g(\bu),\bh(\bu))$, 
we have 
\begin{align} \label{esti:nonlin1a:half}
\begin{aligned}
&\cc{N}_{\hsp}(\bff(\bu),\bg(\bu),g(\bu),\bh(\bu),\bv_0)
\\ &\leq C\sum_{i=0,2}\bigg(
\left\|<t>^b\int_0^t\nabla\bu\otimes(\del_t\bu,\nabla^2\bu)\right\|_{L_{p}(0,T;L_{q_i}(\hsp))}
\\ &\qquad +\sum_{m=1,2}\left\|<t>^b\int_0^t\nabla^m\bu\otimes\nabla\bu\right\|_{L_{p}(0,T;L_{q_i}(\hsp))}
\\ &\qquad +\left\|bt<t>^{b-2}\int_0^t\nabla\bu\otimes\bu\right\|_{L_{p}(0,T;L_{q_i}(\hsp))}
\\ &\qquad +\left\|<t>^b\int_0^t\nabla^2\bu\otimes\del_t\bu\right\|_{L_{p}(0,T;L_{q_0}(\hsp))}
\\ &\qquad +\left\|<t>^b\int_0^t\nabla^2\bu\otimes\del_t\bu\right\|_{L_{p}(0,T;L_{q_2/2}(\hsp))}
\\ &\qquad +\sum_{m=0,1}\|<t>^b\nabla\bu\otimes\nabla^m\bu\|_{L_{p}(0,T;L_{q_i}(\hsp))}
\bigg). 
\end{aligned}
\end{align}
The 0-th derivative of $\bu$ appears only in the terms 
\begin{align}
\left\|bt<t>^{b-2}\int_0^t\nabla\bu\otimes\bu\right\|_{L_{p}(0,T;L_{q_i}(\hsp))}, 
\quad \|<t>^b\nabla\bu\otimes\bu\|_{L_{p}(0,T;L_{q_i}(\hsp))} 
\end{align} 
and the second one is estimated by $C[\bu]_{(0,T)}$ 
from Lemma \ref{Lem:nonlinesti0:half} (a) as in \eqref{esti:duu}. 
The estimate for the first term with $i=0$ is shown from Lemma \ref{Lem:nonlinesti0:half} as 
\begin{align}
&\left\|bt<t>^{b-2}\int_0^t\nabla\bu\otimes\bu\right\|_{L_{p}(0,T;L_{q_0}(\hsp))}
\\ &\leq C\left\|\int_0^t\nabla\bu\right\|_{L_{p}(0,T;L_{q_{03}}(\hsp))}
\|<t>^{b-1}\bu\|_{L_{p}(0,T;L_{q_{04}}(\hsp))}
\leq C[\bu]_{(0,T)}^2
\end{align} 
and the estimate with $i=2$ is obtained 
by replacing $q_{03}$ and $q_{04}$ with $q_2$ and $\infty$ in this calculation. 
The other terms of the right-hand side in \eqref{esti:nonlin1a:half} 
are estimated by $C[\bu]_{(0,T)}^2$ since we have the same estimates 
for $\del_t\bu$, $\nabla\bu$ and $\nabla^2\bu$ as in Section \ref{sec:gen}, 
see \eqref{eq:estifua}, \eqref{esti:didu}, \eqref{esti:ddidu}, 
\eqref{esti:ddiddu0}, \eqref{esti:ddiddu2}, \eqref{esti:dudu}. 
Then we have the desired estimate \eqref{esti:nonlin1:half}. 

\subti{Proof of \eqref{esti:nonlin2:half}}

In the remainder of this paper, 
we prove that the additional term $[K_0g(\bu)]_{(0,T)}$ is harmless
by showing the estimate \eqref{esti:nonlin2:half}. 

We first show the first inequality and second inequality with $m=0$ in \eqref{esti:nonlin2:half}. 
Since \eqref{eq:DataCCMR:half} with $(G,\bG)=(g(\bu),\bg(\bu))$ implies 
\begin{align}
\|[K_0g(\bu)](t)\|_{L_{q}(\hsp)}\leq C\|[\bg(\bu)](t)\|_{L_{q}(\hsp)} 
\text{ a.e.$t\in(0,T)$, }
\end{align} 
by the definition \eqref{def:nonlin} of the nonlinearities, 
H\"{o}lder's inequality \eqref{eq:Holder} and Lemma \ref{Lem:nonlinesti0:half} (d), 
we have
\begin{align} \label{esti:K0gu}
\begin{aligned}
&\|<t>^{\bmmpq[0]}K_0g(\bu)\|_{L_{\tp_0}(0,T;L_{\tq_0}(\hsp))}
\\ &\leq C\|<t>^{\bmmpq[0]}\bg(\bu)\|_{L_{\tp_0}(0,T;L_{\tq_0}(\hsp))}
\\ &\leq C\left\|<t>^{\bmmpq[0]}\int_0^t\nabla\bu\,ds\otimes\bu\right\|_{L_{\tp_0}(0,T;L_{\tq_0}(\hsp))}
\end{aligned}
\end{align}
and, by \eqref{eq:Holder} again and by Lemma \ref{Lem:nonlinesti0:half}, 
the right-hand side is estimated as 
\begin{align} \label{esti:diany}
\begin{aligned}
&\left\|<t>^{\bmmpq[0]}\int_0^t\nabla\bu\,ds\otimes\bu\right\|_{L_{\tp_0}(0,T;L_{\tq_0}(\hsp))}
\\ &\leq C\left\|\int_0^t\nabla\bu\,ds\right\|_{L_{\infty}(0,T;L_{\infty}(\hsp))}
\|<t>^{\bmmpq[0]}\bu\|_{L_{\tp_0}(0,T;L_{\tq_0}(\hsp))}
\leq C[\bu]_{(0,T)}^2. 
\end{aligned}
\end{align}
Similarly, since \eqref{eq:DataCCMR:half} with $(G,\bG)=(\del_tg(\bu),\del_t\bg(\bu))$ implies 
\begin{align}
\|[\del_tK_0g(\bu)](t)\|_{L_{q}(\hsp)}\leq C\|[K_0\del_tg(\bu)](t)\|_{L_{q}(\hsp)}\leq C\|[\del_t\bg(\bu)](t)\|_{L_{q}(\hsp)} 
\text{ a.e.$t\in(0,T)$, }
\end{align} 
by \eqref{eq:repgut}, 
\begin{align} \label{esti:K0gut}
\begin{aligned}
&\|<t>^{\bmmpq[2]}\del_tg(\bu)\|_{L_{\tp_2}(0,T;L_{\tq_2}(\hsp))}
\\ &\leq C\|<t>^{\bmmpq[2]}\del_t\bg(\bu)\|_{L_{\tp_2}(0,T;L_{\tq_2}(\hsp))}
\\ &\leq C\|<t>^{\bmmpq[2]}\nabla\bu\otimes\bu\|_{L_{\tp_2}(0,T;L_{\tq_2}(\hsp))}
\\ &+ C\left\|<t>^{\bmmpq[2]}\int_0^t\nabla\bu\,ds\otimes\del_t\bu\right\|_{L_{\tp_2}(0,T;L_{\tq_2}(\hsp))}. 
\end{aligned}
\end{align} 
We can estimate the second term by $C[\bu]_{(0,T)}^2$ in the same way as in \eqref{esti:diany}.  
The estimate for the first term can be obtained by \eqref{esti:duu} with $q_i=\tq_2$
combined with $\bmmpq[2]\leq b$ and \eqref{eq:exp_restriction}: 
\begin{align} \label{esti:already}
\begin{aligned}
\left\|<t>^{\bmmpq[2]}\nabla\bu\otimes\bu\right\|_{L_{\tp_2}(0,T;L_{\tq_2}(\hsp))}
\leq \left\|<t>^{b}\nabla\bu\otimes\bu\right\|_{L_{p}(0,T;L_{\tq_2}(\hsp))}
\leq C[\bu]_{(0,T)}^2. 
\end{aligned}
\end{align}

The second inequality with $m=1$ of \eqref{esti:nonlin2:half} is proven as follows: 
by \eqref{esti:K0G2} with $G=g(\bu)$, 
the definition \eqref{def:nonlin} of the nonlinearities, 
H\"{o}lder's inequality \eqref{eq:Holder} and Lemma \ref{Lem:nonlinesti0:half} (d), 
\begin{align} \label{esti:dK0gu}
\begin{aligned}
&\|<t>^{\bmmpq[1]}\nabla K_0g(\bu)\|_{L_{\tp_1}(0,T;L_{\tq_1}(\hsp))}
\\ &\leq C\|<t>^{\bmmpq[1]}g(\bu)\|_{L_{\tp_1}(0,T;L_{\tq_1}(\hsp))}
\\ &\leq C\left\|<t>^{\bmmpq[1]}\int_0^t\nabla\bu\,ds\otimes\nabla\bu\right\|_{L_{\tp_1}(0,T;L_{\tq_1}(\hsp))}
\end{aligned}
\end{align}
and the right-hand side is estimated by $C[\bu]_{(0,T)}^2$ in the same manner as in \eqref{esti:diany}. 

By \eqref{esti:K0G2} with $G=g(\bu)$, $b_2(\tp_2,\tq_2)\leq b$, \eqref{eq:exp_restriction}, 
and estimates \eqref{esti:ddidu} and \eqref{eq:estifua}, 
the case $m=2$ can be shown as 
\begin{align} \label{esti:dK0gu2}
\begin{aligned}
&\|<t>^{\bmmpq[2]}\nabla^2K_0g(\bu)\|_{L_{\tp_2}(0,T;L_{\tq_2}(\hsp))}
\\ &\leq C\|<t>^{\bmmpq[2]}\nabla g(\bu)\|_{L_{\tp_2}(0,T;L_{\tq_2}(\hsp))}
\\ &\leq C\sum_{i=0,2}\|<t>^b\nabla g(\bu)\|_{L_{p}(0,T;L_{q_i}(\hsp))}
\\ &\leq C\bigg(\sum_{i=0,2}\left\|<t>^b\int_0^t\nabla^2\bu\,ds\otimes\nabla\bu\right\|_{L_{p}(0,T;L_{q_i}(\hsp))}
\\ &\qquad +\left\|<t>^b\int_0^t\nabla\bu\,ds\otimes\nabla^2\bu\right\|_{L_{p}(0,T;L_{q_i}(\hsp))}\bigg)
\\ &\leq C[\bu]_{(0,T)}^2. 
\end{aligned}
\end{align}

To summarize, we obtain \eqref{esti:nonlin2:half} 
and then, we can conclude 
the global well-posedness of transformed problem \eqref{LNS} 
and the decay property of the solution in $\BR^N_+$ including $N=3$.   

\section{Conclusion} \label{sec:conclusion}
In study, we have 
proven the global well-posedness for quasi-linear parabolic and 
hyperbolic-parabolic equation with non-homogeneous 
boundary conditions in unbounded domains,
and provided a general framework to solve
such problems. In fact, the free boundary problem treated in
this paper is a typical problem for the quasilinear equations with 
non-homogeneous boundary condition, and we can continue to 
study in the case of two-phase problems such as incompressible-incompressible,
incompressible-compressible, and compressible-compressible 
viscous fluid flows.


\end{document}